\documentclass[12pt]{amsart}

\usepackage[hmargin=2.6cm,bmargin=2.55cm,tmargin=3.05cm]{geometry}
\usepackage[usenames]{color}
\usepackage{graphbox}
\usepackage{subcaption}

\usepackage{mathtools}
\numberwithin{equation}{section}
\mathtoolsset{showonlyrefs,showmanualtags}

\usepackage{amsmath,amssymb,amsthm,amsfonts,enumerate,url,mathbbol,mathrsfs,esint,tikz,graphicx,pifont,float}
\usetikzlibrary{calc}
\usepackage{hyperref}

\usepackage[normalem]{ulem}
\usepackage[utf8]{inputenc}

\theoremstyle{plain}
\newtheorem{theorem}{Theorem}[section]
\newtheorem{lemma}[theorem]{Lemma}

\newtheorem{corollary}[theorem]{Corollary}

\theoremstyle{definition}
\newtheorem{definition}[theorem]{Definition}
\newtheorem{example}[theorem]{Example}

\newtheorem{remark}[theorem]{Remark}

\numberwithin{equation}{section}
\numberwithin{figure}{section}

\newcommand{\R}{\mathbb{R}}
\newcommand{\N}{\mathbb{N}}

\newcommand{\Graph}{\mathcal{G}} 
\newcommand{\EdgeSet}{\mathcal{E}} 
\newcommand{\VertexSet}{\mathcal{V}} 
\newcommand{\EndpointSet}{\mathcal{X}} 
\newcommand{\CutSet}{\mathcal{C}}
\newcommand{\Partition}{\mathcal{P}}

\newcommand{\NeuSet}{\mathcal{V}^1}

\newcommand{\parti}{{\mathcal P}} 
\newcommand{\nenergy[1]}{{\Lambda}^N_{#1}} 
\newcommand{\denergy[1]}{{\Lambda}^D_{#1}} 
\newcommand{\noptenergy[1]}{{\mathcal L}^{N}_{#1}} 
\newcommand{\noptenergylax[1]}{{\mathcal L}^{N}_{#1}}
\newcommand{\doptenergy[1]}{{\mathcal L}^D_{#1}} 

\def\:{\thinspace:\thinspace}

\title[Interlacing inequalities for spectral minimal partitions]{Interlacing and Friedlander-type inequalities for spectral minimal partitions of metric graphs} 



\author[M.~Hofmann]{Matthias Hofmann}
\author[J.~B.~Kennedy]{James B.~Kennedy}

\address{Grupo de F\'isica Matem\'atica, Faculdade de Ci\^encias, Universidade de Lisboa, Campo Grande, Edif\'icio C6, P-1749-016 Lisboa, Portugal}
\email{mhofmann@fc.ul.pt}

\address{Departamento de Matem\'atica \emph{and} Grupo de F\'isica Matem\'atica, Faculdade de Ci\^encias, Universidade de Lisboa, Campo Grande, Edif\'icio C6, P-1749-016 Lisboa, Portugal}
\email{jbkennedy@fc.ul.pt}

\thanks{The authors wish to thank Gregory Berkolaiko and Delio Mugnolo for their thoughtful and helpful comments on an earlier version of this paper, and Andrea Serio and the anonymous referee for many useful suggestions on an intermediate version of it. Both authors were supported by the Funda\c{c}\~ao para a Ci\^encia e a Tecnologia, Portugal, via the project NoDES: Nonlinear Dispersive and Elliptic Systems, reference PTDC/MAT-PUR/1788/2020 (J.B.K.), and the programs ``Investigador FCT'', reference IF/01461/2015 (J.B.K.) and ``Bolseiro de Investiga\c{c}\~ao'',  reference PD/BD/128412/2017 (M.H.). The work of both authors was also partially supported by the COST Association (Action CA18232).}

\begin{document}
\begin{abstract}
We prove interlacing inequalities between spectral minimal energies of metric graphs built on Dirichlet and standard Laplacian eigenvalues, as recently introduced in [Kennedy \textit{et al}, Calc.\ Var.\ PDE \textbf{60} (2021), 61]. These inequalities, which involve the first Betti number and the number of degree one vertices of the graph, recall both interlacing and other inequalities for the Laplacian eigenvalues of the whole graph, as well as estimates on the difference between the number of nodal and Neumann domains of the whole graph eigenfunctions. To this end we study carefully the principle of \emph{cutting} a graph, in particular quantifying the size of a cut as a perturbation of the original graph via the notion of its \emph{rank}. As a corollary we obtain an inequality between these energies and the actual Dirichlet and standard Laplacian eigenvalues, valid for all compact graphs, which complements a version for tree graphs of Friedlander's inequalities between Dirichlet and Neumann eigenvalues of a domain. In some cases this results in better Laplacian eigenvalue estimates than those obtained previously via more direct methods.
\end{abstract}

\maketitle

\section{Introduction}
\label{sec:intro}

The study of spectral minimal partitions on domains, first introduced in \cite{CTV05}, is by now a well-established topic; see \cite{BBN18,BNHe17,BNL17,HHT09} and the references therein. Roughly speaking, the goal is to find the $k$-partition of a domain $\Omega \subset \R^d$, consider some eigenvalue, usually the first eigenvalue of the Dirichlet Laplacian, on each of the partition elements, and then seek the partition which minimizes some functional of these eigenvalues, usually their maximum. These are of interest not least because of a number of connections to the eigenvalues and eigenfunctions of the Dirichlet Laplacian on the whole of the domain $\Omega$, for example via their connection to nodal partitions, the nodal count and Pleijel's theorem on the number of nodal domains of any $k$-th eigenfunction of $\Omega$; see \cite[Section~10.2]{BNHe17} for a recent survey.

On compact metric graphs, the subject is newer and much less well developed. Apart from an earlier work studying nodal partitions in particular \cite{BaBeRaSm12} such spectral minimal partitions have only been systematically investigated very recently, since \cite{KeKuLeMu20}, where well-posedness and basic properties of a number of different spectral partitioning problems were established.

There are arguably two features which set spectral partitions of metric graphs apart from their domain counterparts, both arising from the fact that metric graphs are essentially one-dimensional manifolds with singularities (the vertices): 
\begin{enumerate}
\item far more general spectral functionals can be considered than on domains, including in particular considering the smallest nontrivial eigenvalue of the Laplacian with standard (i.e., natural, Neumann--Kirchhoff, Kirchhoff-continuity) vertex conditions at the boundary of each of the partition elements (or \emph{clusters}) instead of Dirichlet; and 
\item deciding how, and which, vertices may be cut to create the clusters also leads to different problems and different optima.
\end{enumerate}

Just as there is an intimate connection between the Dirichlet problem and the nodal domains of Laplacian eigenfunctions on the whole object (graph, domain or manifold), the partitions involving standard Laplacians are related to the \emph{Neumann domains} of the eigenfunctions of the whole graph, as introduced and studied recently \cite{AB19,ABBE20} (see below).

In the current context, broadly speaking, our goal is to explore the relationships between some of these different spectral partition problems, in particular as regards the optimal energies (the values of the functionals at the minimizers), both in comparison with each other and with the Laplacian eigenvalues of the whole graph. In doing so we will see strong parallels between these energies and the way eigenvalues (and eigenfunctions) usually behave.

In order to state our results more precisely, we first have to introduce more precisely which problems we will be considering; full details will be given in Section~\ref{sec:definitions}. Given a compact metric graph $\Graph$ and a $k$-partition $\Partition = (\Graph_1,\ldots,\Graph_k)$, $k \geq 1$, of $\Graph$, that is, a $k$-tuple of \emph{clusters} $\Graph_1,\ldots,\Graph_k$, which are closed connected subgraphs of $\Graph$ intersecting each other at at most a finite number of \emph{boundary points}, denote by $\lambda_1 (\Graph_i)$ the first eigenvalue of the Laplacian on $\Graph_i$ with Dirichlet conditions at the boundary points and standard conditions elsewhere, and by $\mu_2 (\Graph_i)$ the first nontrivial eigenvalue of the Laplacian with standard conditions everywhere; then it follows from \cite{KeKuLeMu20} that the following min-max problems are well posed,\footnote{Here we will be considering fewer classes of partitions, and fewer types of spectral energy, than in \cite{KeKuLeMu20}; we will thus adopt simpler notation. In the language and notation of \cite{KeKuLeMu20} we are interested in the most general case of \emph{connected partitions}, not necessarily exhaustive, and the case $p=\infty$. See Appendix~\ref{appendix:classes} for more details.}
\begin{displaymath}
\begin{aligned}
\doptenergy[k](\Graph)&= \inf\, \{\max \{\lambda_1 (\Graph_1),\ldots,\lambda_1(\Graph_k)\}: \Partition = (\Graph_1,\ldots,\Graph_k) \text{ $k$-partition of } \Graph \}\\
\noptenergy[k](\Graph)&= \inf\, \{\max \{\mu_2 (\Graph_1),\ldots,\mu_2(\Graph_k)\}: \Partition = (\Graph_1,\ldots,\Graph_k) \text{ $k$-partition of } \Graph \}
\end{aligned}
\end{displaymath}
($N$ for natural/Neumann--Kirchhoff), that is, that there always exists a minimizing $k$-partition $\Partition$. It turns out that these spectral minimal energies are in some ways good surrogates for the eigenvalues of Laplacian operators on the whole graph -- even in the ``$N$'' case -- , and certainly better than spectral minimal energies on domains vis-\`a-vis the domain eigenvalues. For example, it was shown in \cite{HKMP20} that such spectral minimal energies satisfy the same Weyl asymptotics as the eigenvalues of the Laplacian with standard vertex conditions on the graph, as well as a number of two-sided bounds strongly reminiscent of similar bounds on the graph eigenvalues (as, for example, may be found in \cite[Section~4]{BKKM17}). It is thus also natural to compare these quantities, both with each other and with Laplacian eigenvalues of the whole graph.

Our principal objective here is to establish sharp interlacing inequalities linking the quantities $\doptenergy[k]$ and $\noptenergylax[k]$: here and throughout we will suppose $\Graph$ to be a fixed connected, compact, finite metric graph (again, see Section~\ref{sec:definitions} for more details); $\beta$ will denote the first Betti number of $\Graph$, i.e., the number of independent cycles in the graph, and $|\NeuSet|$ the number of vertices of $\Graph$ of degree $1$, the \emph{leaves}.

\begin{theorem}\label{thm:intromain1}
For all $k \geq \max\{\beta,1\}$ we have
\begin{displaymath}
\noptenergylax[k](\Graph)\ge \doptenergy[k+1-\beta](\Graph).
\end{displaymath}
\end{theorem}

\begin{theorem}\label{thm:intromain2}
For all $k \geq \beta+|\NeuSet| \geq 1$ we have
\begin{displaymath}
	\doptenergy[k](\Graph)\ge \noptenergylax[k+1-\beta  - |\NeuSet|](\Graph).
\end{displaymath}
\end{theorem}

A consequence of these inequalities is that we can relate these spectral minimal energies with the eigenvalues of the Laplacian on the whole graph, both with standard conditions at all vertices and with Dirichlet conditions at all vertices. Indeed, denote by $\mu_k (\Graph)$ the $k$-th eigenvalue of the Laplacian with standard conditions on $\Graph$ (starting at $\mu_1 (\Graph) = 0$ and counting multiplicities) and $\lambda_k (\Graph, \VertexSet^D)$ the $k$-th eigenvalue of the Laplacian with Dirichlet conditions at a distinguished set $\VertexSet^D$ of Dirichlet vertices and standard conditions on the rest, which we abbreviate to $\lambda_k (\Graph) := \lambda_k (\Graph, \VertexSet)$ for when \emph{all} vertices are Dirichlet vertices. Then the following result is a fairly direct consequence of Theorem~\ref{thm:intromain1}.

\begin{corollary}\label{thm:intromainRohleder}
Let $\Graph$ be a (connected, compact, finite) metric graph with first Betti number $\beta \geq 0$. Then for all $k \geq \beta + 1$ we have
\begin{equation}
\label{eq:intromainRohleder}
	\lambda_k (\Graph) \ge  \noptenergylax[k](\Graph) \ge \doptenergy[k+1-\beta](\Graph) \ge \mu_{k+1-\beta} (\Graph).
\end{equation}
\end{corollary}

This, and indeed the principle of interlacing inequalities between such minimal energies, have several natural motivations. For one, rather suprisingly, combining Corollary~\ref{thm:intromainRohleder} with an upper bound on $\noptenergylax[k]$ obtained in \cite{HKMP20} results in the following bound which, even as a bound on $\mu_k$, actually turns out to be better for many classes of graphs than the central bound \cite[Theorem~4.9]{BKKM17}, as we shall see below.

\begin{corollary}\label{thm:firstapplication}
Let $\Graph$ be a metric graph with first Betti number $\beta\in \mathbb N_0$ and total length $L$, and suppose there exist $n \leq |\EdgeSet|$ \emph{Eulerian paths} covering $\Graph$, crossing at at most finitely many points. Then for all $k \geq \max\{n + 1  -\beta,1\}$ we have
\begin{equation}
\label{eq:firstapplication}
	 \mu_k (\Graph) \leq \doptenergy[k](\Graph) \le \frac{\pi^2}{L^2} (k+n+\beta-2)^2.
\end{equation}
\end{corollary}

Similarly, Theorem~\ref{thm:intromain2} combined with a sharp lower bound on $\noptenergylax[k]$ from \cite{HKMP20} leads to a lower bound on $\doptenergy[k]$ better than existing ones in many cases.

\begin{corollary}
\label{cor:secondapplication}
Let $\Graph$ be a metric graph with first Betti number $\beta\in \mathbb N_0$, $|\NeuSet|$ vertices of degree one, and total length $L$. Then for all $k \geq \beta+|\NeuSet| \geq 1$ we have
\begin{equation}
\label{eq:secondapplication}
	\doptenergy[k](\Graph) \geq \frac{\pi^2}{L^2} (k + 1 - \beta - |\NeuSet|)^2.
\end{equation}
\end{corollary}

But at a more fundamental level a key motivation for Theorems~\ref{thm:intromain1} and~\ref{thm:intromain2} arises from the effect on graph Laplacians, of so-called \emph{surgery} on the graph. Our method of proof of these two theorems, which involves studying \emph{cuts} of a graph and the impact this has on being able to glue together eigenfunctions on different parts of the graph, is intimately related to both the nodal count (and distribution of the nodal domains) of the eigenfunctions, and the number and distribution of the corresponding Neumann domains.

Cutting a graph at a point is a simple operation which changes the topology of the graph and which has a predictable effect on the eigenvalues of the graph, as it represents a finite rank perturbation of the associated Laplacian (see \cite[Sections~3.1 and~4.1]{BKKM19}). Cutting a graph at exactly the points $x$ where the $k$-th standard Laplacian eigenfunction $u_k(x)$ equals $0$ leads to a \emph{nodal partition}, the clusters of which are the nodal domains of the eigenfunction. The number and distribution of these has been explored at some length; see for example \cite{ABB18,BaBeRaSm12,Ber08}. The Neumann domains arise as the clusters of a partition cut at the points where $u_k'(x)=0$; the number of Neumann domains behaves similarly as a function of $k$, at least in the ``generic'' case where (among other things) all cuts are made away from the vertices \cite{AB19}. Perhaps most notably for us, it has been shown in the generic case that the difference between the number of nodal domains $\nu(k)$ and the number of Neumann domains $\xi(k)$ of $u_k$ satisfies \emph{exactly} the same bounds as the indices appearing in Theorems~\ref{thm:intromain1} and~\ref{thm:intromain2} \cite[Proposition~3.1(1)]{AB19} (see also \cite[Proposition~11.2]{ABBE20}):
\begin{equation}
\label{eq:motivationalnodalstuff}
	1 - \beta \leq \nu(k) - \xi(k) \leq \beta + |\NeuSet| - 1.
\end{equation}
In fact, \eqref{eq:motivationalnodalstuff} can also be recovered from our proofs (see Remarks~\ref{rmk:onepartoftheinequalityinintro} and~\ref{rmk:theotherinequalityforintro}). 
Despite the completely different approaches (here we study cutting and pasting eigenfunctions arising from different minimal partitions, in \cite{AB19} the point of departure being the whole graph eigenfunctions) this hints at a much deeper connection between these spectral minimal partitions and the nodal and Neumann domain patterns of the whole graph eigenfunctions, analogous to or extending the connection between nodal domains and partitions explored in \cite{BaBeRaSm12}, which will be left to future investigation to explore fully. 

Somewhat related is the idea of changing a vertex condition from standard to Dirichlet (or vice versa), another finite rank perturbation which leads to interlacing inequalities between Dirichlet and standard Laplacian eigenvalues. A consequence of the min-max characterization of the eigenvalues is the interlacing inequality which in the notation introduced above reads
\begin{displaymath}
	\lambda_{k+|\VertexSet^D|}(\Graph, \VertexSet^D)\ge \mu_{k+|\VertexSet^D|}(\Graph)\ge \lambda_k(\Graph, \VertexSet^D)
\end{displaymath}
(again, see \cite[Section~3.1]{BKKM19}, or, e.g., \cite[Section~3.1.6]{BeKu13}). This is reminiscent of, or rather actually at odds with, Friedlander's inequalities between Dirichlet and Neumann Laplacian eigenvalues on domains in $\R^d$, $d \geq 2$ (see \cite{Fri91}), which assert that $\lambda_k (\Omega) \geq \mu_{k+1} (\Omega)$ for all $k \in \N$; in fact the inequality was later shown to be always strict \cite{Fil05}. Similar results also can be obtained for compact manifolds (see \cite{AM12}). On metric graphs this is rather difficult to recover precisely because of the interlacing inequalities, or the related idea that the difference between Dirichlet and standard vertex conditions is somehow ``smaller'' than the difference between Dirichlet and Neumann boundary conditions. Our Corollary~\ref{thm:intromainRohleder} is a complement to the inequality proved in \cite[Theorem~4.1]{Roh17} for \emph{tree graphs} $\Graph$ (i.e., with $\beta = 0$), which states that
\begin{equation}
\label{eq:rohleder}
	\lambda_k (\Graph) \geq \mu_{k+1} (\Graph), \qquad k \in \N,
\end{equation}
if we impose Dirichlet conditions on all vertices $\VertexSet^D = \VertexSet$ of $\Graph$. Note, however, that \eqref{eq:rohleder} actually holds under the weaker assumption that Dirichlet conditions only be imposed on the leaves of the tree, with standard conditions at all other vertices; see \cite[Lemma~4.5]{BBW15} (with $t=0$).

This paper is organized as follows. We give our notation and introduce the basic notions we will need, such of those of cuts and partitions, in Section~\ref{sec:definitions}. This also entails a rather detailed study of the notion of the cut of a graph; to this end we will introduce the \emph{rank} of a cut, which allows to quantify the corresponding finite rank perturbation of the corresponding function spaces considered. A description of said function spaces, as well as the spectral quantities we will be considering, is given in Section~\ref{sec:func-spec}.  Section~\ref{sec:Proofinterlacing} is devoted to the proof of Theorem~\ref{thm:intromain1}, Section~\ref{sec:proof2} to the proof of Theorem~\ref{thm:intromain2}; in both cases, at the beginning of the section we include a somewhat less formal explanation of where the respective indices appearing in the inequalities come from. In Section~\ref{sec:spectralinequalities} we prove Corollaries~\ref{thm:intromainRohleder},~\ref{thm:firstapplication} and~\ref{cor:secondapplication}, give several examples of graphs where the bounds in \eqref{eq:firstapplication} and \eqref{eq:secondapplication} are better than bounds obtained elsewhere, and also study the case of certain \emph{windmill graphs}, introduced in \cite{KuSe18}, where there is equality everywhere in \eqref{eq:firstapplication}. Finally, in Appendix~\ref{appendix:classes}, we put our results (in particular as regards the types of partitions and energies considered) into the framework of \cite{KeKuLeMu20}.

\section{Metric graphs and partitions}
\label{sec:definitions}
\subsection{Basic assumptions}
\label{sec:basic}
We start with our formalism for metric graphs. Throughout the article we will adopt the framework used in \cite{KeKuLeMu20} and \cite{Mug19}, though it will be necessary to recall some of the particulars here. For us, a \emph{metric graph} $\Graph=(\VertexSet, \EdgeSet)$ will consist of a finite union $\EdgeSet=\{[x_1, x_2], \cdots, [x_{2n-1}, x_{2n}]\}$ of closed, bounded intervals in $\R$, turned into a compact metric space by gluing the intervals at the \emph{endpoint set} $\EndpointSet=\{x_1, \ldots, x_{2n}\}$ in an appropriate sense,\footnote{Thus, for us, a \emph{metric graph} is what is usually called a \emph{compact metric graph} in the literature.} via a partition $\VertexSet =\{v_1, \ldots, v_m\}$ on $\EndpointSet$,
\begin{displaymath}
	\EndpointSet = v_1 \sqcup \ldots \sqcup v_m.
\end{displaymath}
We call each element of $\VertexSet$, which is formally a set of endpoints, a \emph{vertex} of $\Graph$; we call $\VertexSet$ the \emph{vertex set} of $\Graph$ and $\EdgeSet$ the \emph{edge set} of $\Graph$, whose elements, the edges, will be denoted by $e \in \EdgeSet$. The \emph{degree} of a vertex $v_i$ is its cardinality as a set of endpoints; we denote by $\NeuSet \subset \VertexSet$ the set of all vertices of degree one, the \emph{leaves}.

The natural metric on $\Graph$ arises by identifying each vertex with a point, treating each edge $e \in \EdgeSet$ as a subset of $\Graph$, and introducing paths between pairs of points on the graph in accordance with this identification (cf.\ \cite{Mug19}). The length of an edge $e$, i.e., the length of the interval to which it corresponds, will be denoted by $|e|$; the \emph{total length} of $\Graph$ will be denoted by
\begin{displaymath}
	L:= |\Graph|= \sum_{e\in \EdgeSet} |e|.
\end{displaymath}
We say $\Graph$ is connected if it is connected as a metric space, that is, if the distance between any two points in the graph is finite.

We can define an equivalence relation on the class of all such metric graphs via isometric isomorphisms, bijective mappings between graphs which preserve the metric; if two graphs are isometrically isomorphic to each other, then we are in one or both of the following situations:
\begin{enumerate}[(i)]
\item the edge and vertex sets of one graph are permutations (i.e.\ relabelings) of the edge and vertex sets, respectively, of the other;
\item the graphs differ by the presence of dummy vertices, i.e. vertices of degree $2$ that can be added at will essentially subdividing an interval in two intervals of total length of the original interval.
\end{enumerate} 
We will always identify graphs that are isometrically isomorphic, and choose a convenient representative of the corresponding equivalence classes (called \emph{ur-graphs} in \cite{KeKuLeMu20}) in any given context, without further comment.

\subsection{Cuts of a graph}
The notion of \emph{cutting a graph} will be used extensively throughout the paper. While it is by no means new -- among other things it has appeared frequently in the context of spectral geometry of graphs as a prototypical ``surgery principle'' (see, for example, \cite{BKKM19} and the references therein) and was also used in \cite{KeKuLeMu20} as the basis for defining partitions of graphs -- we will need to study this notion far more carefully than in those works, and introduce a number of new concepts around it. We thus start with the basic definition.

\begin{definition}
\label{def:cut}
Let $\Graph$ and $\Graph'$ be metric graphs. Then $\Graph'$ is a \emph{cut} (or \emph{cut graph}) of $\Graph$ if
\begin{enumerate}[(i)]
\item $\Graph$ and $\Graph'$ have a common edge set, and
\item for  all $v'\in \VertexSet'$ there exists $v\in \VertexSet$ such that $v' \subset v$.
\end{enumerate}
We say $v\in \VertexSet$ is a \emph{cut vertex} if there exists $v'\in \VertexSet'$ such that $v'\subsetneq v$, and denote the set of cut vertices, the \emph{cut set}, by $\CutSet(\Graph': \Graph)$, which we treat as a subset of $\Graph$.  If $\CutSet(\Graph':\Graph) = \{v\}$, then we say $\Graph$ has been cut at $v$.
\end{definition}

In practice this definition allows for cutting through the interior of edges of $\Graph$, as in accordance with our observation at the end of Section~\ref{sec:basic} we may always insert dummy vertices at the cut locations before making the cut. The process of ``undoing'' a cut, i.e., reverting to $\Graph$ from $\Graph'$, will as usual be called gluing. In particular, we say $\Graph$ has been obtained from $\Graph'$ by gluing the vertices $v_1,\ldots,v_n \in \VertexSet (\Graph')$ if $\Graph'$ is a cut of $\Graph$ such that $\CutSet(\Graph':\Graph) = \{v\}$, where $v = v_1 \cup \ldots \cup v_n$.

The next notion, while simple will be central for all our interlacing results; intuitively, by quantifying the ``size'' of a cut (see Figure~\ref{fig:cutgraphsexamples}) it gives us direct control on the codimensions of the spaces of functions defined on the respective graphs.

\begin{figure}[h]
	\centering
	\begin{minipage}{.3\textwidth}
	\includegraphics{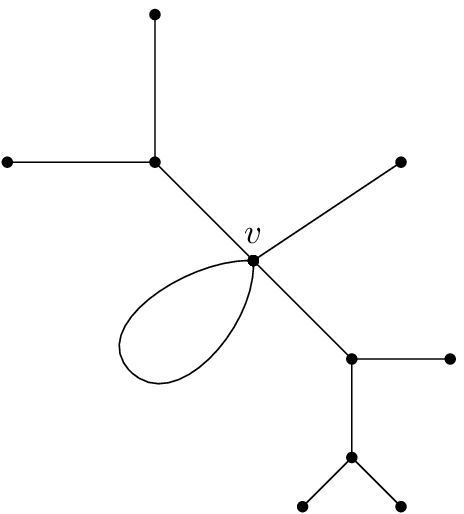}\\ \centering{\textbf{(a)}}
	\end{minipage} \hfill \begin{minipage}{.3\textwidth}
	\includegraphics{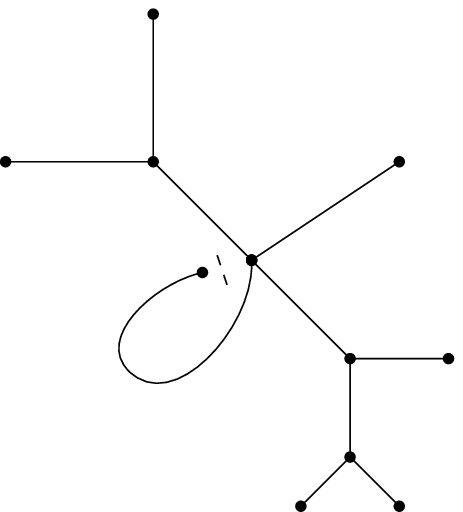}\\ \centering{\textbf{(b)}}
	\end{minipage} \hfill \begin{minipage}{.3\textwidth}
	\includegraphics{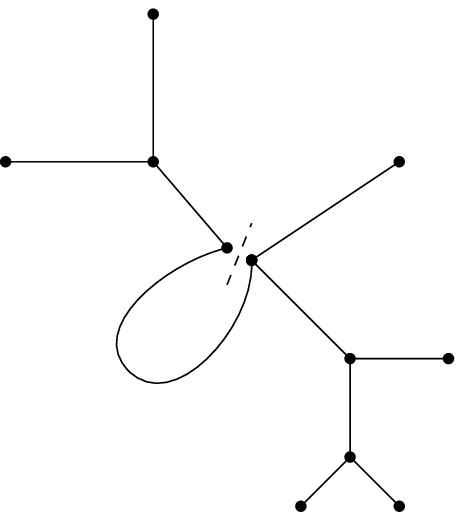}\\ \centering{\textbf{(c)}}
	\end{minipage} \\ \vspace{2em}
	\begin{minipage}{.3\textwidth}
	\includegraphics{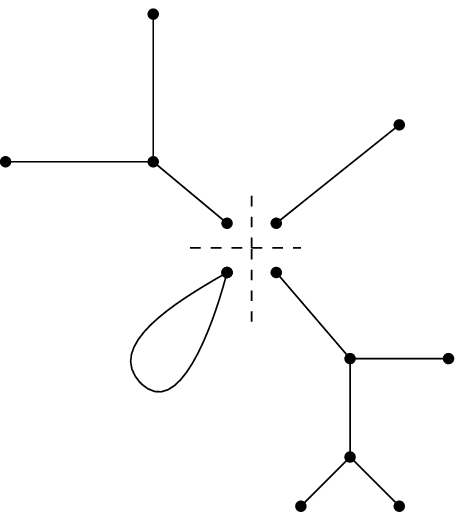}\\ \centering{\textbf{(d)}}
	\end{minipage} \hfill 
	\begin{minipage}{.3\textwidth}
	\includegraphics{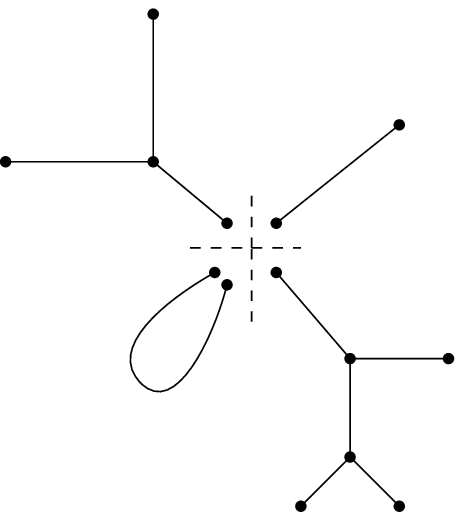}\\ \centering{\textbf{(e)}}
	\end{minipage} \hfill
	\begin{minipage}{.3\textwidth}
	\includegraphics{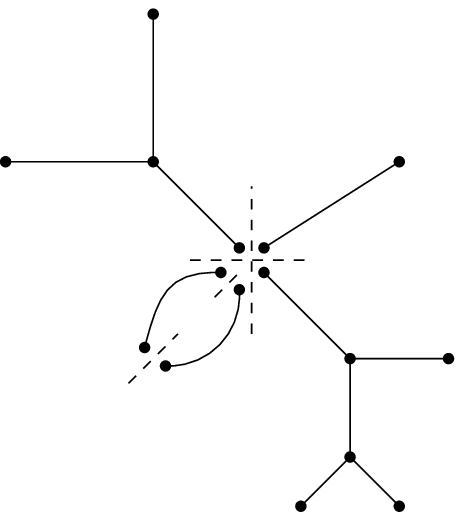}\\ \centering{\textbf{(f)}}
	\end{minipage}
	\caption{Examples of cut graphs of a given graph (a). In (b)--(e) the cut is only at the central vertex $v$: (b) and (c) depict two distinct simple cuts, in (d) the cut has rank three, in (e) it has rank four. Observe that (f), which is a cut of (a) of rank five, is a simple cut of (e), which is a simple cut of (d), which is a cut of neither (b) nor (c).}
	\label{fig:cutgraphsexamples}
\end{figure}

\begin{definition}
Let $\Graph$, $\Graph'$ be metric graphs. Suppose $\Graph'$ is a cut of $\Graph$, such that the graphs have vertex sets $\VertexSet'$ and $\VertexSet$, respectively, then we say
\begin{enumerate} [(i)]
\item $\Graph'$ is a cut of $\Graph$ of \emph{rank}
\begin{displaymath}
\operatorname{rank}(\Graph': \Graph) := |\VertexSet'| - |\VertexSet|.
\end{displaymath} 
\item $\Graph'$ is a \emph{simple cut} if
\begin{displaymath}
\operatorname{rank}(\Graph': \Graph) =1,
\end{displaymath}
i.e. there exists a unique $v\in \VertexSet$ and $v_1', v_2'\in \VertexSet'$ such that $v=v_1' \cup v_2'$ (see also \cite[Definition 2.7(1) with $k=1$]{KeKuLeMu20}).
\end{enumerate}
\end{definition}

\begin{remark}
The rank, thus defined, is invariant under relabeling of the edges and insertion or removal of dummy vertices in $\Graph$ (which by definition of a cut must then also be inserted or removed simultaneously in $\Graph'$), and is hence invariant under isometric isomorphisms of the graph.
\end{remark}

\begin{lemma}
\label{lem:usefultrivialities}
Let $\Graph_1, \Graph_2, \Graph_3$ be metric graphs with common edge set. Suppose $\Graph_1$ is a cut of $\Graph_2$ and $\Graph_2$ is a cut of $\Graph_3$, then 
\begin{enumerate}[(1)]
\item $\Graph_1$ is a cut of $\Graph_3$ and
\begin{displaymath}
\operatorname{rank}(\Graph_1:\Graph_3)= \operatorname{rank}(\Graph_1: \Graph_2) + \operatorname{rank}(\Graph_2:\Graph_3);
\end{displaymath}
\item if  $\operatorname{rank}(\Graph_1:\Graph_3)= \operatorname{rank}(\Graph_2: \Graph_3)$, 
then $\Graph_1=\Graph_2$.
\end{enumerate} 
\end{lemma}

\begin{proof}
Suppose $\Graph_1, \Graph_2, \Graph_3$ have a common edge set, and vertex sets $\VertexSet_1, \VertexSet_2, \VertexSet_3$ respectively, then (1) follows immediately from the definitions of cut and rank. Now suppose that $\operatorname{rank}(\Graph_1:\Graph_3)=\operatorname{rank}(\Graph_2:\Graph_3)$, then 
$
\operatorname{rank}(\Graph_1:\Graph_2)=0
$
and so $k:=|\VertexSet_1|=|\VertexSet_2|$. Let 
\begin{displaymath}
	\VertexSet_1=\{v_1^{(1)}, \ldots, v_k^{(1)}\}, \qquad
	\VertexSet_2=\{v_1^{(2)}, \ldots, v_k^{(2)}\}.
\end{displaymath}
Since $\Graph_1$ is a cut of $\Graph_2$ we may assume, possibly after a relabeling, that $v_i^{(1)}\subset v_i^{(2)}$ for all $i=1,\ldots, k$. But since there is a bijection between the two vertex sets there must be equality, $ v_i^{(1)}= v_i^{(2)}$ for all $i=1,\ldots, k$.
\end{proof}

\begin{remark}
\label{rem:cutpartialordering}
A graph $\Graph_1$ being a cut of $\Graph_2$ defines a partial ordering on the set of metric graphs. Given a metric graph $\Graph$ (with a \emph{fixed} vertex set, i.e., where we do not permit the insertion or removal of dummy vertices), the set of its cut graphs becomes a partially ordered family, and by Lemma~\ref{lem:usefultrivialities} the rank is additive on this family.
\end{remark}

\begin{lemma}\label{lem:usefultounderstand}
Let $\Graph, \Graph'$ be metric graphs. Then $\Graph'$ is a cut of $\Graph$ of rank $k\in \mathbb N$ if and only if there exists a sequence of cuts of $\Graph$
\begin{displaymath}
	\Graph= \Graph^{(0)}\;,\;  \Graph^{(1)}\;,\; \cdots\;,\;\Graph^{(k-1)}\;,\; \Graph^{(k)}=\Graph' 
\end{displaymath}
such that $\Graph^{(i+1)}$ is a simple cut of $\Graph^{(i)}$ for all $i=0,\ldots, k-1$.
\end{lemma}
\begin{proof}
Suppose $\Graph'$ is a cut of $\Graph$ of rank $k\in \mathbb N$, where the graphs have common edge set $\EdgeSet$ and vertex sets $\VertexSet'$, $\VertexSet$, respectively. For the ``only if'' statement we give a constructive proof. Let $v\in \CutSet(\Graph':\Graph)$ and $v'\in \VertexSet'$ such that $v'\subsetneq v$, then we define
\begin{displaymath}
	\VertexSet^{(1)}= \big(\VertexSet\setminus\{v\}\big) \cup\{v', v\setminus v'\}.
\end{displaymath}
Then by construction $\Graph^{(1)}=(\VertexSet^{(1)},\EdgeSet)$ is a simple cut of $\Graph$ and $|\VertexSet^{(1)}|=|\VertexSet|+1$ and $\Graph^{(1)}$ is a simple cut of $\Graph$. One easily sees that $\Graph'$ is a cut of $\Graph^{(1)}$ and by Lemma~\ref{lem:usefultrivialities} (1) we have
\begin{displaymath}
	\operatorname{rank}(\Graph': \Graph^{(1)})= k-1.
\end{displaymath} 
We sucessively construct metric graphs $\Graph^{(1)}, \ldots, \Graph^{(k)}$ such that $\Graph^{(i+1)}$ is a simple cut of $\Graph^{(i)}$ and $\operatorname{rank}(\Graph^{(i)}:\Graph)= i$ for all $i=1,\ldots, k$.  Then $\operatorname{rank}(\Graph':\Graph)=\operatorname{rank}(\Graph^{(k)}: \Graph)$ and so by Lemma~\ref{lem:usefultrivialities} (2) we conclude that $\Graph^{(k)}=\Graph'$.
The other direction is a direct consequence of the additivity of the rank in the sense of Lemma~\ref{lem:usefultrivialities}~(1).
\end{proof}

Note that there is no ``global'' maximal cut of a graph in the sense of the partial ordering as described in Remark~\ref{rem:cutpartialordering}. However, we will have repeated need of the idea of the cut of maximal possible rank \emph{at a given set of vertices}, which replaces each vertex $v$ in that set, say of degree $d_v$, with $d_v$ vertices of degree one. We will call this the \emph{maximal cut} of the graph at the set in question; for example, the maximal cut of the graph in Figure~\ref{fig:cutgraphsexamples}(a) at the vertex $v$ is exactly the graph in~(e). Formally, it may be defined as follows.

\begin{definition}
\label{def:totalcut}
Let $\Graph$ be a metric graph and $\VertexSet_0\subset \VertexSet$ a distinguished vertex set. Then we call the graph $\Graph_1$ with common edge set and vertex set
\begin{displaymath}
	\VertexSet_1 := \big(\VertexSet\setminus\VertexSet_0\big) \cup \bigcup_{v\in \VertexSet_0} \bigcup_{x\in v} \{x\}
\end{displaymath}
the \emph{maximal cut (graph)} of $\Graph$ at $\VertexSet_0$.
\end{definition}

We finish this subsection with another elementary but useful property relating cuts to the first Betti number $\beta = |\EdgeSet| - |\VertexSet| + 1$ of $\Graph = (\VertexSet,\EdgeSet)$. Since here we explicitly allow cuts at the vertices, as is not always the case, we include the short proof.
\begin{lemma}
\label{lem:betacut}
Suppose $\Graph$ has first Betti number $\beta \geq 0$ and $\Graph'$ is a cut of $\Graph$ such that $\operatorname{rank}(\Graph': \Graph) \geq \beta + 1$. Then $\Graph'$ is disconnected.
\end{lemma}

\begin{proof}
By Lemma~\ref{lem:usefultounderstand} there exists an intermediate cut $\Graph_0 = (\VertexSet_0,\EdgeSet)$ of $\Graph$ of rank $\beta$, such that $\Graph'$ is a cut of $\Graph_0$ of rank $\geq 1$. Suppose $\Graph_0$ is not already disconnected (if it is, then so is $\Graph'$). Then $\Graph_0$ has first Betti number $|\EdgeSet| - |\VertexSet_0| + 1 = (|\EdgeSet| - |\VertexSet| + 1) - (|\VertexSet_0| - |\VertexSet|) = \beta-\beta = 0$. So $\Graph_0$ is necessarily a tree, and any further cut will disconnect it. Hence $\Graph'$ cannot be connected.
\end{proof}

\subsection{Partitions}
\label{sec:partitions}

We can now introduce the central objects of study here. The next definition follows \cite[Section~2]{KeKuLeMu20}.

\begin{definition}
Let $k\ge 1$ and let $\Graph=(\VertexSet, \EdgeSet)$ be a metric graph. Then:
\begin{enumerate}[(i)]
\item $\Partition:=(\Graph_1, \ldots, \Graph_k)$ is a \emph{$k$-partition} of $\Graph$ if there exists a cut $\Graph'$ such that $\Graph_1, \ldots, \Graph_k$ are connected components of $\Graph'$.
We refer to the components $\Graph_1,\ldots,\Graph_k$ as \emph{clusters};
\item $\Partition=(\Graph_1, \ldots, \Graph_k)$ is an \emph{exhaustive $k$-partition} if $\Graph'= \sqcup_{i=1}^k \Graph_i$ is a cut graph of $\Graph$ and $\Graph_1, \ldots, \Graph_k$ are its connected components.
\end{enumerate}
We denote the class of all $k$-partitions of $\Graph$ by $\mathfrak C_k(\Graph)$.
\end{definition}

What we are calling $k$-partitions were called \emph{connected} $k$-partitions in \cite{KeKuLeMu20}; this was the broadest class considered in that paper. Note that a $k$-partition need not be exhaustive: the cut $\Graph'$ may have other connected components besides $\Graph_1,\ldots,\Graph_k$, which are not included in $\Partition$. As such, in principle there could be multiple cuts of $\Graph$ which generate $\Partition$ if $\Partition$ is not exhaustive, since one may cut anywhere outside $\sqcup_{i=1}^k \Graph_i$ and still obtain $\Partition$.

For example, if we consider the $3$-partition $\Partition = (\Graph_1,\Graph_2,\Graph_3)$ depicted in Figure~\ref{fig:nonexhaustivepartitionexample} (right) of the graph $\Graph$ of Figure~\ref{fig:cutgraphsexamples}(a) which excludes the loop, then there are multiple possible ways to cut $\Graph$ to extract these three clusters: the graphs of Figure~\ref{fig:cutgraphsexamples}(d), (e) and (f) are all examples.

\begin{figure}[ht]
	\begin{minipage}{.35\textwidth}\includegraphics{examplegraph.eps} \end{minipage}\hspace{3em}
	\begin{minipage}{.35\textwidth}\includegraphics{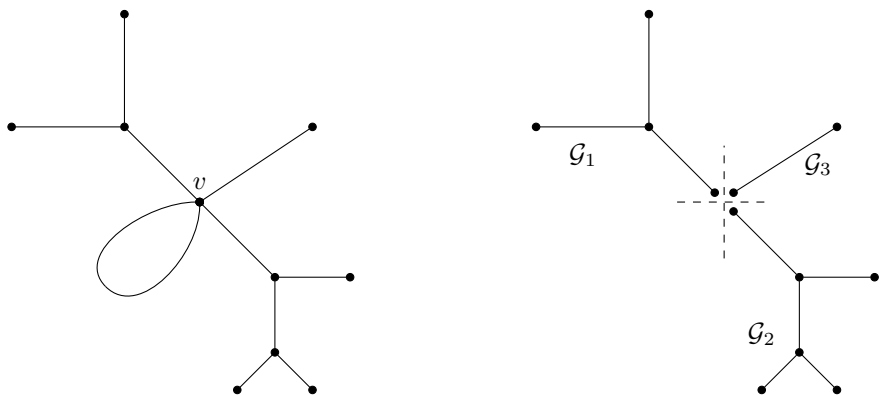} \end{minipage}
	\caption{An example of a non-exhaustive partition (right) of the graph in Figure~\ref{fig:cutgraphsexamples}(a) (reproduced here, left).}
	\label{fig:nonexhaustivepartitionexample}
\end{figure}

However, there will always be a cut of \emph{minimal rank} which gives rise to $\Partition$, as we simply avoid cutting through parts of the graph which are being excluded from the partition whenever possible; we will call this cut graph the \emph{minimal cut} associated with $\Partition$. In our example it is given by (d). In this sense there is no ``maximal cut'' \emph{associated with $\Partition$;} in particular, such a minimal cut is not the direct opposite of the maximal cut \emph{at a vertex set}.

\begin{definition}
\label{def:canonicalcutgraph}
Let $k\ge 1$, $\Graph$ be a metric graph and $\Partition=(\Graph_1, \ldots, \Graph_k)$ be a $k$-partition of $\Graph=(\VertexSet,\EdgeSet)$. Suppose $\Graph_1=(\VertexSet_1,\EdgeSet_1), \ldots, \Graph_k=(\VertexSet_k,\EdgeSet_k)$ with disjoint subsets $\EdgeSet_1, \ldots, \EdgeSet_k$ and $\VertexSet_1, \ldots, \VertexSet_k$ of $\EdgeSet$ and $\VertexSet$ respectively. Then we define the \emph{minimal cut (graph)} $\Graph_{\Partition}$ of $\Graph$ associated with the partition $\Partition$ as the unique cut graph of $\Graph$ of minimal rank such that $\Graph_1, \ldots, \Graph_k$ are connected components of the cut graph. We will refer to the quantity $\operatorname{rank}(\Graph_\Partition:\Graph)$ as the \emph{rank of the partition $\Partition$}.
\end{definition}

To summarize, given any cut there are, in general, multiple partitions which arise from it (any subset of its connected components will be a partition); on the other hand, given a partition, there will, in general, be multiple possible cuts that produce it. Only in the case of exhaustive partitions is there a bijective correspondence between cuts and partitions. Note in Figure~\ref{fig:cutgraphsexamples} that if we consider $4$-partitions of $\Graph$, then (d) and (e) correspond to two distinct partitions.

\begin{remark}
Let $\Graph$ be a metric graph and let $\Partition$ be a $k$-partition, $k \geq 1$. Then, keeping the notation of Definition~\ref{def:canonicalcutgraph}, the minimal cut graph $\Graph_\Partition$ can be constructed as the unique cut graph with the following properties:
\begin{enumerate}[(i)]
\item for any cut vertex $v \in \CutSet(\Graph_\Partition : \Graph)$ there exists at least one cluster $\Graph_i = (\VertexSet_i,\EdgeSet_i)$ and $v_i \in \VertexSet_i$ such that $v_i \subset v$;
\item if a cut vertex $v \in \CutSet(\Graph_\Partition : \Graph)$ is not divided among vertices of the $\Graph_i$, that is, if $w := v \setminus \bigcup_{\tilde{v} \in \bigcup_{i=1}^k \VertexSet_i}$ is non-empty, then there is exactly one connected component of $\Graph_\Partition$ such that $w$ is a vertex of that connected component.
\end{enumerate}
Hence the minimal cut graph $\Graph_{\Partition}$ may be described formally as the metric graph with the same edge set as $\Graph$ and vertex set
\begin{displaymath}
	\VertexSet_\Partition = \bigcup_{i=1}^k \VertexSet_i \cup 
	\left\{ v \setminus \bigcup_{\tilde{v} \in \bigcup_{i=1}^k \VertexSet_i} \tilde v : v \in \VertexSet \setminus \bigcup_{i=1}^k\VertexSet_i \right\}.
\end{displaymath}
\end{remark}

We next introduce the following notation for the boundary points of a partition.

\begin{definition}
\label{def:cutandboundary}
Let $\Partition = (\Graph_1, \ldots, \Graph_k)$ be a (not necessarily exhaustive) $k$-partition of $\Graph$, $k\geq 1$, and let $\Graph_\Partition$ be the minimal cut graph of $\Graph$ associated with $\Partition$.
\begin{enumerate}[(i)]
\item We say $\Graph_i=(\VertexSet_i, \EdgeSet_i)$ and $\Graph_j=(\VertexSet_j, \EdgeSet_j)$, $i\neq j$, are \emph{neighboring clusters}, or just \emph{neighbors}, if there exist $v_i\in \VertexSet_i$, $v_j\in \VertexSet_j$ and $v\in \VertexSet$ such that $v_i, v_j\subset v$. We call any such $v\in \VertexSet$ a \emph{boundary vertex} (or \emph{boundary point}) of $\Partition$, and define the \emph{boundary set of $\Partition$} to be the set of all such boundary points:
\begin{displaymath}
\partial \parti := \{v\in \VertexSet : \exists v_i \in \VertexSet_i, v_j\in \VertexSet_j, i\neq j: v_i, v_j \subset v\}.
\end{displaymath}
\item We define the \emph{boundary set of the cluster $\Graph_i = (\VertexSet_i,\EdgeSet_i)$} by
\begin{displaymath}
\partial \Graph_i =\{ v\in \VertexSet_i : \exists v'\in \CutSet(\Graph_\Partition: \Graph): v\subsetneq v'\}
\end{displaymath}
(cf.\ Definition~\ref{def:cut}).
\end{enumerate}
\end{definition}

For the partition $\Partition = (\Graph_1,\Graph_2,\Graph_3)$ from Figure~\ref{fig:nonexhaustivepartitionexample}, all three clusters are neighbors of each other; $v \in \VertexSet$ is the unique boundary vertex of $\Partition$, and the boundary sets of the $\Graph_i$ consist, respectively, of the three degree-one vertices obtained after cutting $\Graph$ at $v$.

Sometimes it is convenient to be able to consider only exhaustive partitions, especially when dealing with properties of neighboring clusters. To this end, and to complete the picture as regards cuts versus partitions, we introduce the following notion.

\begin{definition}
\label{def:exhaustive-extension}
Let $\Graph$ be a metric graph and let $\Partition=\{\Graph_1, \ldots, \Graph_k\}$ be a (non-exhaustive) $k$-partition of $\Graph$, $k\geq 1$, with edge sets $\EdgeSet_1, \ldots, \EdgeSet_k\subset \EdgeSet$. We say that $\Partition'=\{\Graph_1', \ldots, \Graph_k'\}$ is an \emph{exhaustive extension} of $\Partition$ if
\begin{enumerate}[(i)]
\item $\Graph_{\Partition}$ is a cut of $\Graph_{\Partition'}$
\item $\EdgeSet_i \subset \EdgeSet_i'$ for all $i=1, \ldots, k$
\item $\bigcup_{i=1}^k \EdgeSet_i'= \EdgeSet$.
\end{enumerate}
\end{definition}

The exhaustive $4$-partitions corresponding to the cuts in Figure~\ref{fig:cutgraphsexamples}(d) and (e), and the exhaustive $5$-partition in (f), are all exhaustive extensions of the partition $\Partition = (\Graph_1,\Graph_2,\Graph_3)$ depicted in Figure~\ref{fig:nonexhaustivepartitionexample}.

We finish this section with a useful estimate.

\begin{lemma}\label{lem:usefulestimate}
Let $\Graph$ be a metric graph with first Betti number $\beta \geq 0$. 
Suppose $\Partition=(\Graph_1, \ldots, \Graph_k)$ is an exhaustive $k$-partition of $\Graph$, $k\geq 1$. Then
\begin{equation}
\label{eq:rankestimate}
k-1 \le \operatorname{rank}(\Graph_\Partition: \Graph) \le k-1+\beta .
\end{equation}
\end{lemma}
\begin{proof}
Assume without loss of generality that $\Graph_1=(\VertexSet_1, \EdgeSet_1)$ and $\Graph_2=(\VertexSet_2, \EdgeSet_2)$ are neighboring clusters. Let $v_1 \in \VertexSet_1$, $v_2 \in \VertexSet_2$ and $v\in \VertexSet$ such that $v_1,  v_2\subset v$. We glue $\Graph_1$ and $\Graph_2$ at $v$, i.e. we obtain a graph $\Graph^{(1)}$ with edge set $\EdgeSet^{(1)} =\EdgeSet_1 \cup \EdgeSet_2$ and vertex set  $\VertexSet^{(1)}= \VertexSet_1 \cup \VertexSet_2 \cup \{v_1 \cup v_2\} \setminus \{v_1, v_2\}$. By construction $\Partition^{(1)}=\{\Graph^{(1)}, \Graph_3, \ldots, \Graph_k\}$ defines a $k-1$-partition and  $\Graph_\Partition$ is a simple cut of $\Graph_{\Partition^{(1)}}$. Applying this procedure iteratively and invoking Lemma~\ref{lem:usefultounderstand}, we end up with an exhaustive $1$-partition $\Graph^{(k-1)}$ such that $\Graph_{\Partition}$ is a cut of $\Graph^{(k-1)}$, $\Graph^{(k-1)}$ is a cut of $\Graph$, and $\operatorname{rank}(\Graph_\Partition: \Graph^{(k-1)})=k-1$. Lemma~\ref{lem:usefultrivialities} now yields the lower bound in \eqref{eq:rankestimate}.

On the other hand, since $\Graph$ admits $\beta$ independent cycles any cut of rank $\beta+1$ would necessarily disconnect $\Graph$ (see Lemma~\ref{lem:betacut}); since $\Graph^{(k-1)}$ is a connected cut graph of $\Graph$ we thus have $\operatorname{rank}(\Graph^{(k-1)}:\Graph)\le \beta$. Lemma~\ref{lem:usefultrivialities} now yields the upper bound in \eqref{eq:rankestimate}.
\end{proof}

\section{Function spaces and spectral functionals}
\label{sec:func-spec}

Given a metric graph $\Graph=(\VertexSet, \EdgeSet)$ the spaces $C(\Graph), L^2(\Graph)$ of continuous and square integrable functions, and the Sobolev space $H^1(\Graph)$, respectively, may be defined in the usual way, cf.\ \cite{Mug19}:
\begin{displaymath}
\begin{gathered}
	C(\Graph)=\left \{f\in \bigoplus_{e\in \EdgeSet} C(e) : f(x)= f(y) \text{ for all } x,y\in v, \text{ for all } v\in \VertexSet \right \},\\
	L^2(\Graph) =\bigoplus_{e\in \EdgeSet} L^2(e), \qquad H^1(\Graph) = C(\Graph) \cap \bigoplus_{e\in \EdgeSet} H^1(e).
\end{gathered}
\end{displaymath}
We recall that we treat each vertex $v$ as a point in $\Graph$; in a slight but common abuse of notation we write $f(v)$ for the common value that $f \in C(\Graph)$ takes at all endpoints $x \in v$, and, accordingly, regard $f$ as a function on $\Graph$. Given a distinguished set of vertices $\VertexSet^D$ (which may include dummy vertices) we also define 
\begin{displaymath}
	H_0^1(\Graph, \VertexSet^D) =\{f\in H^1(\Graph): f(v)=0 \text{ for all } v\in \VertexSet^D \}.
\end{displaymath}
If $\Graph'$ is a cluster of a partition of $\Graph$ and $f \in H^1_0 (\Graph',\partial\Graph')$, then we will identify $f$ with a function in $H^1(\Graph)$ in the canonical way, by extension to zero outside $\Graph'$. We also record the following notions related to the zero set of an $H^1$-function for later use.

\begin{definition}
Given $f\in H^1(\Graph)$,
\begin{enumerate}
\item[(i)] its nodal set is defined to be
\begin{displaymath}
	N(f)=\{x\in \Graph: f(x)=0\},
\end{displaymath}
and we call $x \in N(f)$ a \emph{nodal point} (of $f$);
\item[(ii)] we call each connected component of $\Graph \setminus N(f)$ a \emph{nodal domain} of $f$;
\item[(iii)] we say a nodal point $x\in N(f)$ is \emph{non-degenerate} if $f\not \equiv 0$ in any neighborhood of $x$.
\end{enumerate}
\end{definition}

Our notion of \emph{non-degenerate} nodal point is relatively weak: we do not require $f$ to change sign in a neighborhood of $x$, the purpose of the definition is merely to exclude any regions of the graph where $f$ vanishes identically. However, whenever $f$ is a Laplacian eigenfunction (in practice the only case of interest; see below) it is easy to see, by invoking either the maximum principle or the fact that its restriction to each edge is a trigonometric function, that it does in fact change sign in any neighborhood of any non-degenerate nodal point. More than that, in this case its set of non-degenerate nodal points is finite; thus we may assume without loss of generality that they are vertices.
 
A \emph{nodal partition} is the non-exhaustive partition of $\Graph$ whose clusters are the (closures of the) nodal domains of $f$. More precisely, its clusters are those connected components on which $f$ does not vanish identically, of the maximal cut graph of $\Graph$ at all non-degenerate nodal points.

We will be interested in the smallest nontivial eigenvalue of the Laplacian subject to Dirichlet conditions on $\VertexSet^D$ and standard, a.k.a.\ natural (continuity plus Kirchhoff) conditions at all other vertices. If $\VertexSet^D =\emptyset$ this eigenvalue may be described variationally by
\begin{equation}
\label{eq:mu2}
	\mu_2(\Graph) = \inf\left \{ \frac{\int_{\Graph} |f'(x)|^2\, \mathrm dx}{\int_{\Graph} |f(x)|^2\, \mathrm dx} :
	\; f\in H^1(\Graph)\setminus \{0\} \text{ such that } \int_{\Graph} f\, \mathrm dx =0\right \} ;
\end{equation}
we will also, very loosely, refer to this problem as the ``Neumann case'', or $N$ case.\footnote{Note, however, that \emph{Neumann} vertex conditions (not to be confused with Neumann--Kirchhoff) are different from standard ones; the former do not require continuity at the vertex, while at the vertex \emph{every} derivative on every edge needs to be zero.} If at least one vertex is equipped with Dirichlet conditions, then the smallest nontrivial eigenvalue is
\begin{equation}
\label{eq:lambda1}
	\lambda_1(\Graph,\VertexSet^D) =\inf \left \{ \frac{\int_{\Graph} |f'(x)|^2\, \mathrm dx}{\int_{\Graph} |f(x)|^2\, \mathrm dx} :
	\; f\in H^1_0(\Graph, \VertexSet^D)\setminus \{0\}\right \}
\end{equation}
and we talk of the Dirichlet case. Here, if the Dirichlet vertex set $\VertexSet^D$ is clear we will just write $\lambda_1 (\Graph)$. In both cases equality is attained exactly at the corresponding eigenfunctions. As usual we will call the ratio appearing in \eqref{eq:mu2} and \eqref{eq:lambda1} the Rayleigh quotient of $f$.

As in \cite{KeKuLeMu20}, given a $k$-partition $\Partition=(\Graph_1, \ldots, \Graph_k)$ of $\Graph$ we consider the spectral energies
\begin{displaymath}
\begin{gathered}
	\nenergy[k](\Partition):= \max\{ \mu_2 (\Graph_1), \ldots, \mu_2 (\Graph_k)\}, \\
	\denergy[k](\Partition):= \max\{\lambda_1 (\Graph_1), \ldots, \lambda_1 (\Graph_k)\},
\end{gathered}
\end{displaymath}
where unless explicitly stated otherwise we will \emph{always} take the Dirichlet eigenvalue $\lambda_1 (\Graph_i)$ with zero set $\VertexSet_D = \partial\Graph_i$ being the boundary set of the cluster. The subscript $k$ will always denote the number of clusters of the partition; if for a given partition this number is unknown, then we will omit it and simply write $\nenergy[](\Partition)$, $\denergy[](\Partition)$.\footnote{As mentioned in the introduction, here and below our notation is slightly different from that in \cite{KeKuLeMu20}, as here the subscript gives the number of clusters $k$ rather than the value of $p$ in the $p$-norm of the combination of eigenvalues. The reason is that here it will be important to keep track of the number of clusters of our partitions, whereas we always consider functionals built via the $\infty$-norm of the vector of eigenvalues. We are also suppressing a superscript; see Appendix~\ref{appendix:classes}.}

The $k$-partitions minimizing the energies $\nenergy[k]$, $\denergy[k]$ among all (suitable) $k$-partitions are called \emph{spectral minimal $k$-partitions}; the concrete minimization problems we will consider here are to find
\begin{displaymath}
	\noptenergylax[k](\Graph) = \inf_{\Partition \in \mathfrak C_k} \nenergy[k](\Partition), \qquad {\doptenergy[k]}(\Graph) = \inf_{\Partition \in \mathfrak C_k} \denergy[k](\Partition).
\end{displaymath}
In both cases the results of \cite[Section~4]{KeKuLeMu20} guarantee the existence of $k$-partitions attaining these respective infima among all \emph{exhaustive} $k$-partitions. Here we note that it makes no difference for these two minimization problems whether one restricts to exhaustive partitions or not, as will follow from the next lemma (in this context see Definition~\ref{def:exhaustive-extension}).

\begin{lemma}
\label{lem:Partitionexh}
Let $\Graph$ be a metric graph and let $\Partition \in \mathfrak C_k (\Graph)$ be a $k$-partition of $\Graph$, $k\ge 1$. Then there exists an exhaustive extension $\Partition'\in \mathfrak C_k(\Graph)$ such that
\begin{displaymath}
	\nenergy[k](\Partition) \ge \nenergy[k](\Partition') \qquad \text{and} \qquad
	\denergy[k](\Partition) \ge \denergy[k](\Partition').
\end{displaymath}
\end{lemma}

\begin{proof}
We deal with the $N$ case; in the Dirichlet case the argument is the same. Suppose $\Partition=(\Graph_1, \ldots, \Graph_k)$ and let $u_1, \ldots, u_k$ be any eigenfunctions associated with $\mu_2(\Graph_1),\ldots, \mu_2(\Graph_k)$, respectively. Suppose the minimal cut graph $\Graph_{\Partition} = (\VertexSet_\Partition,\EdgeSet)$ associated with $\Partition$ has connected components
\begin{displaymath}
	\Graph_1, \ldots, \Graph_k, \Graph_{k+1}, \ldots, \Graph_{m},
\end{displaymath}
where we assume that $m>k$ since otherwise the partition $\Partition$ would be exhaustive and there would be nothing to show. Let $\Graph_i=(\VertexSet_i,\EdgeSet_i)$ for all $i=1,\ldots, m$.

Let $v\in \partial\Partition$ be a boundary vertex between some $\Graph_i$, $i\leq k$, and $\Graph_j$, $k<j \leq m$ (which exists since by assumption the partition is not exhaustive, so that the respective unions of the $\Graph_i$, $i\leq k$ and the $\Graph_j$, $j>k$, are both nonempty, and $\Graph$ is connected, so these unions share a boundary), so that there exist $v_i\in \VertexSet_i$ for some $i\le k$ and $v_j\in \VertexSet_j$ for some $k<j\le m$  such that $v_i, v_j\subset v$. We define a (first) extension $\Partition^{(1)}=(\Graph_1, \ldots, \Graph_{i-1},\Graph_i^{(1)},\Graph_{i+1},\ldots,\Graph_k)$, where $\Graph_i^{(1)}$ is the (connected) union of $\Graph_i$ and $\Graph_j$, glued at $v_i$ and $v_j$: in particular, the minimal cut graph $\Graph_{\Partition^{(1)}}$ has edge set $\EdgeSet$ and vertex set
\begin{displaymath}
	\VertexSet_{\Partition^{(1)}}= \VertexSet_{\Partition} \setminus \{v_i, v_j\} \cup \{v_i \cup v_j\};
\end{displaymath}
Since $\Graph_i$ and $\Graph_j$ are glued at a single vertex of $\Graph_i^{(1)}$, \cite[Theorem~3.4]{BKKM19} yields $\mu_2(\Graph_i^{(1)}) \leq \mu_2 (\Graph_i)$ and hence $\nenergy[k](\Partition^{(1)})\le \nenergy[k](\Partition)$. We successively construct partitions $\Partition^{(\ell+1)}$ from $\Partition^{(\ell)}$ such that $\nenergy[k](\Partition^{(\ell+1)})\le \nenergy[k](\Partition^{(\ell)})$. Then $\Partition' := \Partition^{(m-k)}$ is necessarily exhaustive, and $\nenergy[k](\Partition') \le \nenergy[k](\Partition)$.
\end{proof}

\begin{corollary}
\label{cor:existenceof}
For all $k\geq 1$ we have
\begin{displaymath}
\begin{aligned}
	\noptenergylax[k](\Graph) &= \inf \{\nenergy[k](\Partition): \Partition \text{ is an \emph{exhaustive} $k$-partition of $\Graph$} \},\\
	\doptenergy[k](\Graph) &= \inf \{\denergy[k](\Partition): \Partition \text{ is an \emph{exhaustive} $k$-partition of $\Graph$} \}.
\end{aligned}
\end{displaymath}
\end{corollary}

To repeat: when minimizing the energies $\nenergy[k]$, $\denergy[k]$ among all $k$-partitions, it makes no difference whether we minimize over \emph{all} $k$-partitions, or over all \emph{exhaustive} $k$-partitions.

\section{Proof of Theorem~\ref{thm:intromain1}}
\label{sec:Proofinterlacing}

Assume $\Graph$ is a connected metric graph with first Betti number $\beta= |\EdgeSet|-|\VertexSet|+1$.  We first wish to give an intuitive explanation as to why Theorem~\ref{thm:intromain1} should hold. So let $\Partition=(\Graph_1, \ldots, \Graph_k)\in \mathfrak C_k(\Graph)$ be an exhaustive $k$-partition realizing the minimum for $\noptenergylax[k](\Graph)$. 
Consider any eigenfunctions $u_1, \ldots, u_k$ on $\Graph_1,\ldots,\Graph_k$ associated with $\mu_2 (\Graph_1),\ldots, \mu_2 (\Graph_k)$, respectively. We extend each function by zero to obtain an $L^2$-function on $\Graph$, which can also be treated as an element of $\bigoplus H^1(e)$, and which we still denote by $u_i$, $i=1,\ldots,k$. Now since each of these functions necessarily changes sign the sets
\begin{displaymath}
u_{i,+} = \mathbb{1}_{\{u_{i}>0\}} u_{i}, \qquad u_{i,-} = \mathbb{1}_{\{u_{i}<0\}} u_{i} ,
\end{displaymath}
$i=1,\ldots, k$, are all non-empty. Suppose we can \emph{match} these eigenfunctions at the cut vertices in the sense that there exist $\alpha_{1,+}, \alpha_{1,-}, \ldots, \alpha_{k,+}, \alpha_{k,-} \in \mathbb R\setminus \{0\}$ such that for all $v\in \CutSet(\Graph_{\Partition}:\Graph)$ 
\begin{equation}
\label{eq:matching}
\alpha_{i, \operatorname{sign}(u_i(v_1))} u_i(v_1) = \alpha_{j, \operatorname{sign}(u_j(v_2))} u_j(v_2)
\end{equation}
for all $v_1, v_2\in \CutSet_v(\Graph_{\Partition})$. 
Then 
\begin{equation}
\label{eq:udef}
u:=\alpha_{1,+} u_{1,+}+ \alpha_{1,-} u_{1,-} + \ldots + \alpha_{k,+} u_{k,+} + \alpha_{k,-} u_{k,-}\in H^1(\Graph).
\end{equation}
How many nodal domains will $u$ have on $\Graph$? We know that:
\begin{enumerate}
\item regarded as a function on the cut graph $\Graph_\Partition$ it has at least $2k$, since it changes sign on each connected component $\Graph_i$, $i=1,\ldots,k$, of $\Graph_\Partition$;
\item by Lemma~\ref{lem:usefulestimate} we have
\begin{displaymath}
	k-1\le \operatorname{rank}(\Graph_\Partition:\Graph)\le k-1+\beta;
\end{displaymath}
\item every time we make a cut of $\Graph$ of rank $1$ (cf.\ Lemma~\ref{lem:usefultounderstand}) the number of nodal domains of $u$ considered as a function on the cut graph increases by at most $1$.
\end{enumerate}
It follows that $u\in H^1(\Graph)$ admits at least $2k-\operatorname{rank}(\Graph_\Partition:\Graph)\ge k+1-\beta$ nodal domains; moreover, its Rayleigh quotient on each of these nodal domains will be no larger than $\nenergy[k](\Partition) = \max_i \mu_2 (\Graph_i)$. Thus, \emph{if} $u \in H^1 (\Graph)$, then we can use the associated nodal partition to obtain
\begin{displaymath}
\noptenergy[k](\Graph) \ge {\noptenergylax[k]}(\Graph) \ge \doptenergy[2k-\operatorname{rank}(\Graph_\Partition:\Graph)](\Graph) \ge \doptenergy[k+1-\beta](\Graph),
\end{displaymath}
which is Theorem~\ref{thm:intromain1}. But of course in general we cannot expect the matching conditions \eqref{eq:matching} to hold.

\begin{figure}[ht]
	\begin{tikzpicture}[scale=1.2]
		\coordinate (c) at (0,0);
		\foreach \i in {1,2,3} {
			\coordinate (v\i) at (120*\i:0.5);
			\coordinate (u\i) at (120*\i:0.7);
			\coordinate (w\i) at (120*\i:1.7);
			\coordinate (x\i) at (120*\i:1.9);
			\coordinate (y\i) at (120*\i:2.4);
			\coordinate (f\i) at ($(120*\i:.6)+({120*\i+90}:.3)$);
			\coordinate (g\i) at ($(120*\i:.6)+({120*\i-90}:.3)$);
			\coordinate (d\i) at ($(120*\i:1.8)+({120*\i+90}:.3)$);
			\coordinate (e\i) at ($(120*\i:1.8)+({120*\i-90}:.3)$);
			\draw[thick] (c) -- (v\i) (u\i)--(w\i) (x\i)--(y\i);
			\draw[dashed, thick] (f\i) -- (g\i) (d\i) -- (e\i);
			\draw[fill] (y\i) circle (1.75pt);
			\draw[fill=white] (v\i) circle (1.75pt) (u\i) circle (1.75pt) (w\i) circle (1.75pt) (x\i) circle (1.75pt);}
		\draw[fill] (c) circle (1.75pt);
	\end{tikzpicture}\hspace{3em}
	\begin{tikzpicture}[scale=1.2]
		\coordinate (c) at (0,0);
		\foreach \i in {1,2, 3} {
		\coordinate (u\i) at (120*\i:0.2);
		\coordinate (w\i) at (120*\i:1.2);
		\coordinate (x\i) at (120*\i:1.4);
		\coordinate (y\i) at (120*\i:2.4);
		\coordinate (f\i) at ($(120*\i:.1)+({120*\i+90}:.3)$);
		\coordinate (g\i) at ($(120*\i:.1)+({120*\i-90}:.3)$);
		\coordinate (d\i) at ($(120*\i:1.3)+({120*\i+90}:.3)$);
		\coordinate (e\i) at ($(120*\i:1.3)+({120*\i-90}:.3)$);
		\draw[thick] (u\i)--(w\i) (x\i)--(y\i);
		\draw[dashed, thick] (f\i) -- (g\i) (d\i) -- (e\i);
		\draw[fill] (y\i) circle (1.75pt);
		\draw[fill] (u\i) circle (1.75pt) (w\i) circle (1.75pt) (x\i) circle (1.75pt);}
	\end{tikzpicture}
	\caption{Spectral minimal partitions for the star graph from Example~\ref{ex:simpleexampleequality}: $k=7$ in the Dirichlet case (left) and $k=6$ in the natural case (right). Open circles represent Dirichlet conditions, solid ones correspond to natural conditions. Note that eigenvalues of all the clusters across both partitions are equal to $4\pi^2$, the first Dirichlet eigenvalue, and first nontrivial Neumann eigenvalue, of an interval of length $\tfrac{1}{2}$.}
	\label{fig:simpleexampleequality}
\end{figure}
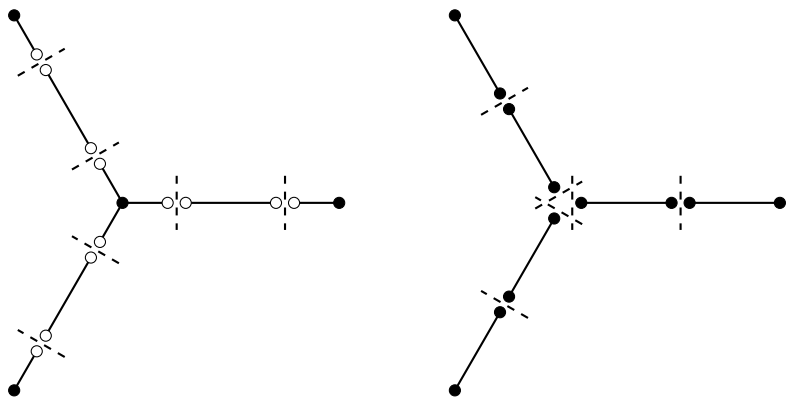

\begin{example}
\label{ex:simpleexampleequality}
Before proceeding,. we give a simple example to show that equality is possible in both Theorem~\ref{thm:intromain1}, that is, that we may have equality
\begin{equation}
\label{eq:simpleexampleequality}
	\noptenergy[k](\Graph) = \doptenergy[k+1-\beta](\Graph),
\end{equation}
as well as in the above argument. Consider the equilateral $m$-star, $m \geq 3$, with $m$ edges of length $1$ each. We identify each edge $e$ with the unit interval $[0,1]$, with $0$ corresponding to the central vertex. As shown in \cite[Lemmata~7.1 and~7.4]{HKMP20} we have
\begin{equation}
	\noptenergy[jm](\mathcal S_m) = \pi^2 j^2= \doptenergy[jm+1](\mathcal S_m).
\end{equation} 
for all $j \geq 1$ (see also Figure~\ref{fig:simpleexampleequality} for a special case of spectral minimal partitions for the $3$-star; in fact, Figure~\ref{fig:simpleexampleequality} is representative for the form of all these spectral minimal partitions). Moreover, in this case there is a nodal partition corresponding to $\doptenergy[jm+1](\mathcal S_m)$, which comes from taking eigenfunctions of the form $u_{e,j} (x) = \cos (\pi j x)$ on each edge $e \simeq [0,1]$. Note that equality need not hold for integers $k$ not of the form $jm+1$, since for example
\begin{equation}
	\noptenergy[jm-1]= \frac{\pi^2 m^2 j^2}{L^2} > \frac{\pi^2 m^2 (j-1/2)^2}{L^2} = \doptenergy[jm]
\end{equation}
for all $j \geq 1$.
\end{example}

Note that \eqref{eq:simpleexampleequality} also holds for the loop and for the interval, for all $k \geq 1$.

\begin{remark}
\label{rmk:onepartoftheinequalityinintro}
Suppose $u$ is an eigenfunction, with eigenvalue $\lambda$, of the (standard) Laplacian on $\Graph$, and suppose that considering the maximal cut of $\Graph$ at all points where $u$ reaches a local nonzero maximum or minimum generates a partition with $k = \xi (u)$ clusters. (In the language of Section~\ref{sec:proof2} and \cite{ABBE20} this means $u$ has $\xi(u)$ \emph{Neumann domains}.) Then $\lambda$ equals the first nontrivial standard Laplacian eigenvalue on each cluster, with eigenfunction $u$ (see \cite[Lemma~8.1]{ABBE20}). Now by construction we can certainly match these restrictions of the eigenfunction at the cut points, in accordance with the above discussion. As we have seen, the resulting Dirichlet partition consists of at least $\xi(u)+1-\beta$ clusters, which in this case are clearly the nodal domains of $u$. Thus we recover one part of \eqref{eq:motivationalnodalstuff}.
\end{remark}

\begin{lemma}
\label{lem:constructiongeneral}
Let $\Graph$ be a metric graph and $\Graph'$ a cut of $\Graph$. Let $r:=\operatorname{rank}(\Graph':\Graph)$ and $k>r$, then for any $k$-partition $\Partition'=(\Graph_1', \ldots, \Graph'_k)$  of $\Graph'$ there exists a $(k-r)$-partition $\Partition=(\Graph_1, \ldots, \Graph_{k-r})$ of $\Graph$ such that
\begin{equation}
\label{eq:dwhattoshow}
	\denergy[k](\Partition') \ge \denergy[k-r](\Partition).
\end{equation}
\end{lemma}

\begin{proof}
By a simple induction argument based on Lemma~\ref{lem:usefultounderstand} it suffices to prove the result for $r=1$. So suppose $\Graph'$ is a simple cut of $\Graph$ and $\Partition'=(\Graph_1', \ldots, \Graph_k')$ is an arbitrary $k$-partition of $\Graph'$. We let $\tilde u_i \in H^1_0 (\Graph_i',\partial\Graph_i')$ be an eigenfunction associated with $\lambda_1(\Graph_i')$, $i=1,\ldots,k$, then the function $\tilde u$ such that $\tilde u|_{\Graph_i'} = \tilde u_i$ for all $i$ belongs to $H^1 (\Graph')$ and has nodal partition exactly $\Partition'$. We show that there exists $u\in H^1(\Graph)$ with at least $k-1$ nodal domains such that, likewise, $u|_{\Graph_i'} = \tilde u_i$ for all $i$. A simple argument using the nodal partition $\Partition$ associated with $u$ and fact that $\denergy[k](\Partition') = \max_i \lambda_1 (\Graph_i')$ leads to $\denergy[k](\Partition') \ge \denergy[k-r](\Partition)$.

To prove the former statement, suppose that $v$ is the unique vertex in $\CutSet(\Graph':\Graph)$, and $v_1, v_2\in \VertexSet'$ such that $v=v_1\cup v_2$. Suppose without loss of generality that $v_1\in \Graph_1$ and $v_2\in \Graph_i$ for some $i=1,\ldots,k$ (where, in particular, $i=1$ is possible).

\emph{Case 1: $\tilde u(v_1) =0$ and $\tilde u(v_2)=0$.}
Then since $\tilde u(v_1)=\tilde u(v_2)$ we infer $\tilde u\in H^1(\Graph)$ and we are done since $\tilde u$ admits $n$ connected nodal domains.

\emph{Case 2: $\tilde u(v_1)\neq 0 \neq \tilde u(v_2)$, and $i \neq 1$.} Then there exist $\alpha_1, \alpha_2\neq 0$ such that
\begin{displaymath}
	\alpha_1 \tilde u(v_1)= \alpha_2 \tilde u(v_2)
\end{displaymath}
and we may define
\begin{displaymath}
	u(x): = \begin{cases}
		\alpha_1 \tilde u(x), \qquad & x\in \Graph_1\\
		\alpha_2 \tilde u(x), \qquad & x\in \Graph_i\\
		\tilde u(x) \qquad &\text{otherwise,}
	\end{cases}
\end{displaymath}
so that $u\in H^1(\Graph)$ with $k-1$ nodal domains.

\emph{Case 3: Otherwise.} Suppose without loss of generality that $\tilde u(v_1)\neq 0$ and $\tilde u(v_2)=0$. Then we define
\begin{displaymath}
	u(x):= \begin{cases}
		\tilde u(x), &\qquad x\not \in \Graph_1\\
		0, &\qquad \text{otherwise,}
	\end{cases}
\end{displaymath}
and by construction $u\in H^1(\Graph)$ has $k-1$ nodal domains.
\end{proof}

\begin{corollary}
\label{cor:constructiongeneral}
Let $\Graph$ be a metric graph and $\Graph'$ a cut of $\Graph$, and suppose $r:=\operatorname{rank}(\Graph':\Graph)$. Then
\begin{displaymath}
	\doptenergy[k](\Graph')\ge  \doptenergy[k-r](\Graph).
\end{displaymath}
\end{corollary}

\begin{lemma}
\label{lem:interlacingestimate1}
Let $\Graph$ be a metric graph and let $\Partition=(\Graph_1, \ldots, \Graph_k)$ be a $k$-partition with minimal cut graph $\Graph_\Partition$. Let $r:= \operatorname{rank}(\Graph_\Partition:\Graph)$, then there exists a $(2k-r)$-partition
\begin{displaymath}
	\Partition'=(\Graph_1, \ldots, \Graph_{2k-r})
\end{displaymath}
such that
\begin{displaymath}
	\nenergy[k](\Partition) \ge \denergy[2k-r](\Partition').
\end{displaymath}
\end{lemma}

\begin{proof}
Suppose $\Graph_\Partition$ is the minimal cut graph of $\Partition=(\Graph_1, \ldots, \Graph_k)$. Let $u_i$ be an eigenfunction for $\mu_2 (\Graph_i)$ on $\Graph_i$, $i=1,\ldots, k$. Then $u_i$ necessarily changes sign on $\Graph_i$ and hence admits at least two nodal domains; denote by $\Partition_i = (\Graph_{i,+},\Graph_{i,-})$ any exhaustive extension of any corresponding nodal $2$-partition of $\Graph_i$. Then, since $\mu_2 (\Graph_i) \geq \max \{\lambda_1(\Graph_{i,+}), \lambda_1 (\Graph_{i,-})\}$,
\begin{displaymath}
	\nenergy[k](\Partition)\ge \max_{i=1,\ldots, k} \max\{ \lambda_1(\Graph_{i,+}), \lambda_1(\Graph_{i,+})\}\ge \doptenergy[2k](\Graph_{\Partition})\ge \doptenergy[2k-r](\Graph),
\end{displaymath}
where the last inequality follows from Corollary~\ref{cor:constructiongeneral}.
\end{proof}

We can now give the proof of Theorem~\ref{thm:intromain1}.

\begin{proof}[Proof of Theorem~\ref{thm:intromain1}]
Let $\Partition$ be any $k$-partition of $\Graph$. By Lemma~\ref{lem:usefulestimate} we have
\begin{displaymath}
	\operatorname{rank}(\Graph_{\Partition}:\Graph) \le k-1+\beta
\end{displaymath}
and so, applying Lemma~\ref{lem:interlacingestimate1}, taking the infimum over all such partitions and using the monotonicity of the mapping $j \mapsto \doptenergy[j] (\Graph)$, we obtain
\begin{displaymath}
	\noptenergylax[k](\Graph)\ge \doptenergy[k+1-\beta](\Graph).
\end{displaymath}
\end{proof}

\section{Proof of Theorem~\ref{thm:intromain2}}
\label{sec:proof2}

Just as the basic idea behind Theorem~\ref{thm:intromain1} is gluing together nodal domains of standard Laplacian eigenfunctions of the partition clusters to construct a test partition, here we will be interested in the ``dual'' problem of gluing together the so-called \emph{Neumann domains} of the cluster Dirichlet eigenfunctions (see, e.g., \cite{AB19,ABBE20}). 

Let us again start with an intuitive explanation of Theorem~\ref{thm:intromain2}. We suppose $\Graph$ is a metric graph with first Betti number $\beta$ and $|\NeuSet|$ leaves. We take $\Partition=(\Graph_1, \ldots, \Graph_k)\in \mathfrak C_k(\Graph)$ to be a fixed exhaustive $k$-partition of $\Graph$ and consider the respective first Dirichlet eigenfunctions $u_1, \ldots, u_k$ on $\Graph_1, \ldots, \Graph_k$, associated with $\lambda_1(\Graph_1), \ldots, \lambda_1(\Graph_k)$ and extended by zero on the rest of $\Graph$.

We decompose each $\Graph_i$ by taking the maximal cut (see Definition~\ref{def:totalcut}) of $\Graph_i$ at every point, without loss of generality a vertex $v \in \VertexSet (\Graph_i)$, at which $u_i$ attains a nonzero extremum, and thus in particular $\frac{\partial}{\partial \nu}|_e u_i (v) = 0$ on every edge $e$ incident with $v$. On each connected component $\widetilde\Graph_{i,1}, \ldots, \widetilde\Graph_{i,k_i}$, $k_i \geq 1$, the Neumann domains, $u_i$ is the first eigenfunction of the Laplacian with suitable mixed Dirichlet-standard conditions, and in particular $\lambda_1 (\Graph_i)$ is still the first eigenvalue of each $\widetilde\Graph_{i,j}$ by a standard variational argument (cf.\ \cite[Proof of Theorem~3.4]{BKKM17}, or also \cite[Lemma~8.1]{ABBE20} for a similar principle).

Now suppose that, given a cut vertex $v \in \CutSet(\Graph_{\Partition}:\Graph)$, we glue together all the neighboring Neumann domains $\Graph_{i_1,j_1}, \ldots, \Graph_{i_{k_v},j_{k_v}}$ at $v$ to form a cluster $\Graph' := \Graph_{i_1,j_1} \cup \ldots \cup \Graph_{i_{k_v},j_{k_v}}$; then, by taking a suitable linear combination of $u_{i_1}|_{\Graph_{i_1,j_1}}, \ldots, u_{i_{k_v}}|_{\Graph_{i_{k_v},j_{k_v}}}$ similarly to \eqref{eq:udef}, we obtain a test function on $\Graph'$, orthogonal to the constant functions for the right choice of coefficients, whose Rayleigh quotient is at most $\max \{\lambda_1 (\Graph_{i_1}), \ldots, \lambda_1 (\Graph_{i_{k_v}}) \} \leq \denergy[k] (\Partition)$.

Gluing such neighboring Neumann domains together at as many different cut vertices as possible (see also Figure~\ref{fig:zigzag}), we may thus construct a partition $\Partition'$ of $\Graph$ such that $\nenergy[] (\Partition') \leq \denergy[k] (\Partition)$.
\smallskip

The question is, how many clusters can $\Partition'$ have? Denote by $\widetilde\Partition = \{\widetilde{G}_{i,j}\}_{i,j}$ the partition of $\Graph$ which results from taking the maximal cut of each of the clusters of $\Partition$ in the sense described above, which will be a finer partition than $\Partition$ and $\Partition'$. We wish to determine how many clusters must be created when passing from $\Partition$ to $\widetilde\Partition$, and how many may be lost from $\widetilde\Partition$ to $\Partition'$.

For the first question, we wish to find a condition that guarantees that a cluster $\Graph_i$ of $\Partition$ will yield (at least) two in $\widetilde\Partition$, that is, that it contains at least two Neumann domains. A sufficient condition is that $\Graph_i$ have at least two Dirichlet (cut) vertices, and that $u_i$ reach an extremum on every \emph{trail} (non-self-intersecting path) in $\Graph_i$ connecting them. Observe that this need not be the case if the cluster contains a leaf or a cycle of $\Graph$ (for example if $\Graph_i$ is an interval with one Dirichlet and one Neumann condition, or lasso with a Dirichlet condition at its degree-one vertex). This motivates the following definition.

\begin{definition}
Suppose $\Partition = (\Graph_1, \ldots, \Graph_k)$ is a $k$-partition, $k \geq 2$, of $\Graph$. We say that a cluster $\Graph_i$ is \emph{benign} (in $\Graph$) if it contains neither a vertex of $\Graph$ of degree one, nor a cycle of $\Graph$; otherwise, we say it is \emph{malign}.
\end{definition}

Observe that any benign cluster of $\Graph$ must necessarily be a tree each of whose leaves belongs to the cut set $\CutSet(\Graph_{\Partition}:\Graph) = \partial\Partition$, while for malign clusters this is not necessarily the case. We see that if $\Partition$ has $k'$ malign clusters, then $\widetilde\Partition$ must have at least $2k-k'$ clusters.

The next question is how many clusters we may lose going from $\widetilde\Partition$ to $\Partition'$; the example of Figure~\ref{fig:zigzag} shows that the answer may be complicated.

\begin{figure}[ht]
\includegraphics[scale=0.7]{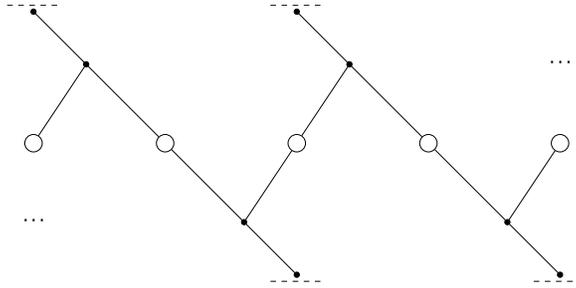}
\caption{A possible cluster of $\Partition'$ resulting by gluing the corresponding Neumann domains at the cut vertices of $\Partition$, which are the Dirichlet points (open circles) of the associated eigenfunctions. Observe that while this cluster is composed of a large number of Neumann domains, being constructed in this way it necessarily contains in its interior at least one boundary point of the original Dirichlet partition.}
\label{fig:zigzag}
\end{figure}

The following lemma formalizes the above reasoning and answers the latter question; the proof of Theorem~\ref{thm:intromain2} will then follow easily.

\begin{lemma}\label{lem:reverseestimate}
Let $\Graph$ be a compact, connected metric graph. Suppose $\Partition=(\Graph_1, \ldots, \Graph_k)$ is an exhaustive $k$-partition, $ k \geq 2$, with
\begin{displaymath}
\operatorname{rank}(\Graph_{\Partition}: \Graph)= k-1+r
\end{displaymath}
for some $0\le r \le \beta$ and suppose that $\Partition$ contains at most $1 \leq n \leq k - r$ malign clusters. Then there exists an exhaustive $k+1-n-r$-partition $\Partition'$ of $\Graph$ such that 
\begin{displaymath}
\denergy[k](\Partition) \ge \nenergy[k+1-n-r](\Partition'). 
\end{displaymath}
\end{lemma}

\begin{proof}
Let $u_1, \ldots, u_k$ be the respective first Dirichlet eigenfunctions on $\Graph_1, \ldots, \Graph_k$, associated with $\lambda_1(\Graph_1), \ldots, \lambda_1(\Graph_k)$, identified as functions in $H^1(\Graph)$ via extension by zero. Take $\Partition'$ to be the partition of $\Graph$ associated with the cut $\Graph_{\Partition'}$ of $\Graph$ consisting of the maximal cut of $\Graph$ at all points where any of the $u_i$ admit a local nonzero extremum.

We construct a new partition $\widetilde\Partition$ which is the coarsest partition of $\Graph$ finer than both $\Partition$ and $\Partition'$: more precisely, we let $\Graph_{\widetilde\Partition}$ be the unique lowest-rank cut of $\Graph$ which is simultaneously a cut of $\Graph_{\Partition}$ and $\Graph_{\Partition'}$ (for example, for the cuts in (b) and (d) in Figure~\ref{fig:cutgraphsexamples} this would be (e)), and we let $\widetilde\Partition$ be the (exhaustive) partition whose clusters are exactly the connected components of $\Graph_{\widetilde\Partition}$. Then by construction
\begin{equation}
\operatorname{rank}(\Graph_{\Partition'}:\Graph) = \operatorname{rank}(\Graph_{\widetilde\Partition}: \Graph)- \operatorname{rank}(\Graph_\Partition:\Graph).
\end{equation}
It follows from Lemma~\ref{lem:usefultrivialities} that
\begin{equation}
\operatorname{rank}(\Graph_{\widetilde{\Partition}}:\Graph_{\Partition'})=\operatorname{rank}(\Graph_{\widetilde{\Partition}}:\Graph)- \operatorname{rank}(\Graph_{\Partition'}:\Graph)=\operatorname{rank}(\Graph_{\Partition}:\Graph),
\end{equation}
that is, $\operatorname{rank}(\Graph_{\widetilde{\Partition}}:\Graph_{\Partition'}) = k - 1 + r$.

\begin{figure}[htp]
\centering
\begin{subfigure}{.33\textwidth}
\centering
\includegraphics[scale=0.8]{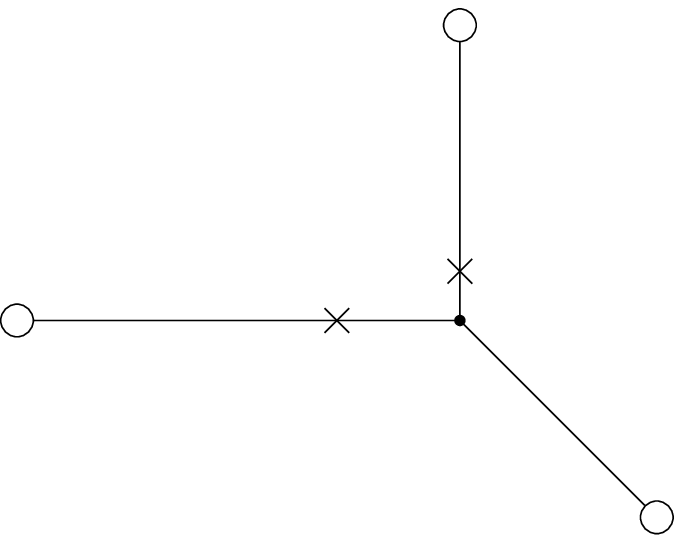}
\end{subfigure}\qquad
\begin{subfigure}{.33\textwidth}
\centering
\includegraphics[scale=0.8]{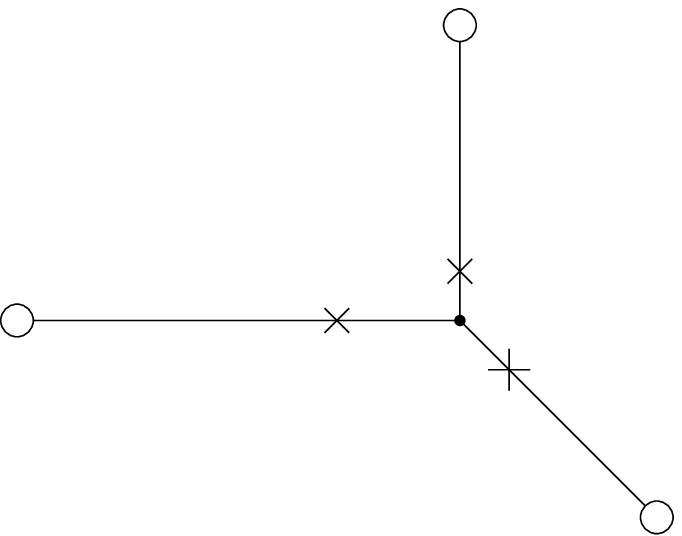}
\end{subfigure}
\caption{On the left is an example of a possible cluster of $\Partition$ with nodes (open circles) and local extrema (crosses) of the eigenfunction. Such an eigenfunction may have multiple local extrema within the cluster. However, the local extrema cannot enclose an area as in the image on the right. In particular, any cluster of $\widetilde{\Partition}$, and thus of $ \Partition'$, necessarily contains a node, that is, a boundary vertex of $\Partition$.}\label{fig:enclosed}
\end{figure}

Next observe that every benign cluster admits at least two Neumann domains and therefore $\widetilde{\Partition}$ has at least $2k-n$ clusters. Lemma~\ref{lem:usefultounderstand} combined with a simple induction argument shows that $\Partition'$ has at least $(2k-n) - (k-1+r) = k+1-n-r$ clusters, since undoing a simple cut (i.e., gluing once) will change the number of connected components of the cut graph by at most one.

We claim that every cluster $\Graph'$ of $\Partition'$ satisfies $\mu_2 (\Graph') \leq \denergy[k](\Partition)$, which will complete the proof of the lemma. To this end, fix such a cluster $\Graph'$ of $\Partition'$; we first observe that $\Graph'$ contains at least two clusters of $\widetilde{\Partition}$, that is, it is formed out of at least two distinct Neumann domains of the eigenfunctions $u_1,\ldots,u_k$ (cf.~also Figure~\ref{fig:zigzag}). To see this, observe that:
\begin{enumerate}[(1)]
\item the boundary sets of $\Partition$ and $\Partition'$ are disjoint: at any cut vertex of $\Graph_\Partition$ all the $u_i \in H^1(\Graph)$ satisfy a Dirichlet condition; hence no such point can also be a local nonzero extremum;
\item by construction, on each cluster of $\widetilde{\Partition}$ there is exactly one eigenfunction $u_i$ which does not vanish identically, and this eigenfunction does not change sign within the cluster;
\item no eigenfunction has a strictly positive local minimum or strictly negative local maximum anywhere; hence no eigenfunction can have a Neumann domain strictly contained in a nodal domain (see also Figure~\ref{fig:enclosed}). 
\end{enumerate}
If $\Graph'$ should coincide with a single cluster $\widetilde\Graph$ of $\widetilde\Partition$, then in particular the two must share a boundary set. This means, by construction of $\widetilde\Partition$ and (1), that $\widetilde\Graph$ contains no Dirichlet points, that is, the only eigenfunction $u_i$ (from (2)) which does not vanish identically in $\widetilde\Graph$, cannot have any zeros whatsoever there. But this is a contradiction to (3).

As a result, we can guarantee the existence of a function $\varphi\in H^1(\Graph')$ such that 
\begin{equation}
\int_{\Graph'} \varphi(x)\, \mathrm dx=0,
\end{equation}
by taking $\varphi$ to be a suitable linear combination of the restrictions of $u_i|_{\Graph'}$, $i=1,\ldots, k$. Then $\varphi$ is a valid test function for $\mu_2(\Graph')$ on the one hand, and on the other the Rayleigh quotient of $\varphi$ cannot exceed $\denergy[k](\Partition) =  \max_{i=1,\ldots,k} \lambda_1 (\Graph_i)$. The latter claim follows from a standard argument: by construction, on every nodal domain of $\varphi$ we have that $\varphi$ is a multiple of some $u_i$, and thus its Rayleigh quotient is no larger than the maximum of the Rayleigh quotients of the $u_i$ on the respective nodal domains. Moreover, $u_i$ satisfies either a standard or a Dirichlet condition at every vertex of this nodal domain, treated as a subgraph of $\Graph'$, and is thus a non-sign-changing classical eigenfunction there, so, as noted earlier, $\lambda_1 (\Graph_i)$ is equal to the Rayleigh quotient of $u_i$ on $\Omega$.
\end{proof}

\begin{proof}[Proof of Theorem~\ref{thm:intromain2}]
The theorem will follow immediately from Lemma~\ref{lem:reverseestimate} once we have shown that any exhaustive partition $\Partition$ of $\Graph$ of rank $k-1+r$, $0 \leq r \leq \beta$, can have at most $n = \beta + |\NeuSet| - r \leq k - r$ malign clusters (recall we are assuming $k \geq \beta + |\NeuSet|$).

We first observe that at most $|\NeuSet|$ clusters can contain at least one leaf of $\Graph$; it remains to show that at most $\beta - r$ clusters can contain a cycle of $\Graph$. But this follows if we can show that the (disconnected) minimal cut graph $\Graph_\Partition$ has Betti number $\beta - r$. This, in turn, follows from a simple induction argument using the definition of $r$ and Lemmata~\ref{lem:usefultounderstand} and~\ref{lem:usefulestimate}: there will exist an intermediate cut of $\Graph$ rank $r$ which remains connected and has Betti number $\beta - r$; $\Graph_\Partition$ is then obtained from this intermediate graph by cutting $k-1$ times in such a way that each cut splits off an additional connected component (cluster of $\Partition$) from the rest of the graph.
\end{proof}

\begin{remark}
\label{rmk:theotherinequalityforintro}
Let $u$ be an eigenfunction of the Laplacian on $\Graph$ and $\Partition$ be its nodal partition with $k=\nu(u)$ nodal domains (clusters). We know that on each the restriction of $u$ to each cluster coincides with the corresponding first eigenfunction on that cluster, with Dirichlet conditions at the boundary points.  Then by construction, the partition $\Partition'$ in Lemma~\ref{lem:reverseestimate} coincides with the partition of $\Graph$ into the Neumann domains of $u$. The proof of Theorem~\ref{thm:intromain2} in particular ensures that this partition contains at least
\begin{equation}
\label{eq:theotherinequalityforintro}
\xi(u) \ge \nu(u)+1-\beta-|\NeuSet|
\end{equation}
clusters, the Neumann domains of $u$. Combining \eqref{eq:theotherinequalityforintro} and Remark~\ref{rmk:onepartoftheinequalityinintro}, we recover \eqref{eq:motivationalnodalstuff}.
\end{remark}

\section{Application: Spectral inequalities}
\label{sec:spectralinequalities}

In this section we will prove Corollary~\ref{thm:intromainRohleder}, relating the interlacing inequalities of Theorems~\ref{thm:intromain1} and~\ref{thm:intromain2} to the eigenvalues of the Laplacian on the whole graph $\Graph$ with Dirichlet and standard vertex conditions. Afterwards, we will discuss their relation with concrete estimates on the optimal energies $\doptenergy[k] (\Graph)$, $\noptenergy[k] (\Graph)$ in terms of geometric and topological properties of $\Graph$ and in particular prove Corollaries~\ref{thm:firstapplication} and~\ref{cor:secondapplication}; complementary estimates were obtained in \cite[Theorems~3.1 and~3.2]{HKMP20}. We recall that $\lambda_k (\Graph,\VertexSet) =: \lambda_k (\Graph)$ and $\mu_k (\Graph)$ are, respectively, the $k$-th eigenvalue, counted with multiplicities, of the Laplacian with Dirichlet conditions at \emph{all} vertices of $\Graph$ (which thus reduces to a disjoint union of $n$ intervals), and of the Laplacian with standard conditions at all vertices of $\Graph$.

\begin{proof}[Proof of Corollary~\ref{thm:intromainRohleder}]
We clearly only have to prove the first and the last inequalities, the middle one being contained in Theorem~\ref{thm:intromain1}. For the first inequality, $\noptenergylax[k](\Graph) \leq \lambda_k (\Graph)$, we observe, firstly, that for any finite interval $I \subset \R$ and $j \in \N$, $\lambda_j (I) = \mu_{j+1} (I)$.

We suppose that for each $i=1,\ldots,n$,
\begin{displaymath}
	j_i := \max \{ j \geq 0 : \lambda_j (e_i) \leq \lambda_k (\Graph)\},
\end{displaymath}
so that the collection $\{\lambda_\ell (e_i) : 1 \leq \ell \leq j_i \}$ gives exactly the first $k$ eigenvalues $\lambda_1 (\Graph), \ldots, \lambda_k (\Graph)$, counted with multiplicities (if $\lambda_k (\Graph)$ is multiple, meaning at least two edges have the same eigenvalue corresponding to $\lambda_k (\Graph)$, then we arbitrarily choose a certain number to be excluded in order to guarantee that $\{\lambda_\ell (e_i) : 1 \leq \ell \leq j_i \}$ does in fact consist of exactly $k$ elements, the largest of which is $\lambda_k (\Graph)$).

For each $i=1,\ldots,n$ for which $j_i \geq 1$, we partition the edge $e_i$ into $j_i$ equal subintervals $e_{i,1},\ldots,e_{i,j_i}$, each of which is a nodal domain for the eigenfunctions of $\lambda_{j_i} (e_i)$, so that, with our first observation, $\mu_2 (e_{i,1}) = \ldots = \mu_2 (e_{i,j_i}) = \lambda_1 (e_{i,1}) = \lambda_{j_i} (e_i)$. Since $\sum_{i=1}^n j_i = k$, the (non-exhaustive) partition $\Partition := \{ e_{i,\ell} : 1 \leq \ell \leq j_i,\, 1 \leq i \leq n\}$ is a $k$-partition of $\Graph$ such that
\begin{displaymath}
	\nenergy[k] (\Partition) = \max_{i,\ell} \mu_2 (e_{i,\ell}) = \max_{i} \lambda_{j_i} (e_i) = \lambda_k (\Graph).
\end{displaymath}
The inequality now follows from Lemma~\ref{lem:Partitionexh}. The last inequality,
\begin{displaymath}
	\mu_k (\Graph) \leq \doptenergy[k](\Graph),
\end{displaymath}
follows from a standard argument involving the min-max characterization of $\mu_k (\Graph)$, see also \cite[Proposition~8.5]{KeKuLeMu20} for a detailed proof.
\end{proof}

\subsection{Upper bounds on $\doptenergy[k]$} We now turn to Corollary~\ref{thm:firstapplication},and more generally to concrete estimates from above on $\doptenergy[k] (\Graph)$ and $\noptenergylax[k] (\Graph)$ in terms of geometric and metric properties of $\Graph$. We recall that we wish to prove \eqref{eq:firstapplication}, which we reproduce here for the sake of convenience:
\begin{equation}
\label{eq:recall-firstapplication}
	\mu_k(\Graph) \leq \doptenergy[k](\Graph) \leq \frac{\pi^2}{L^2}(k+n+\beta-2)^2
\end{equation}
for all $k\ge \max\{n+1-\beta, 1\}$, where $n \geq 1$ is any number such that there exists an $n$-partition of $\Graph$ each of whose clusters consists of a single \emph{Eulerian path} (clearly $n \leq |\EdgeSet|$ since we can always take the $|\EdgeSet|$-partition of $\Graph$ into its individual edges).

\begin{proof}[Proof of Corollary~\ref{thm:firstapplication}]
From \cite[Theorem~5.3]{HKMP20} we have
\begin{equation}\label{eq:hkmp20n}
	\noptenergylax[k] (\Graph) \leq \frac{\pi^2}{L^2}\big(k + n-1\big)^2
\end{equation}
for all $k \geq n$, with $n \geq 1$ as just described. Combining this with Corollary~\ref{thm:intromainRohleder} we obtain
\begin{equation}
	\mu_k(\Graph)\le \doptenergy[k](\Graph)\le \noptenergy[k-1+\beta](\Graph) \le \frac{\pi^2}{L^2} (k+n+\beta-2)^2  
\end{equation}
for $k\ge \max\{n+1-\beta, 1\}$.
\end{proof}

We observe that our inequality \eqref{eq:firstapplication} involves rather different quantities from the upper bound in \cite[Theorem~5.1]{HKMP20}
\begin{equation}\label{eq:hkmp20d}
	\mu_k(\Graph) \le \doptenergy[k](\Graph) \le \frac{\pi^2}{L^2} \left ( k+ \left ( |\EdgeSet| -1 - \left\lfloor\frac{|\NeuSet|}{2}\right\rfloor \right ) \right ),
\end{equation}
where $|\NeuSet|$ is, as usual, the number of degree one vertices of $\Graph$, as well as what is possibly the best general upper bound on $\mu_k(\Graph)$ to date, namely \cite[Theorem~4.9]{BKKM17}
\begin{equation}
\label{eq:bkkm}
	\mu_k (\Graph) \leq \frac{\pi^2}{L^2}(k+\tfrac{3}{2}\beta+\tfrac{1}{2}|\NeuSet|-2)^2
\end{equation}
for all $k\geq 1$ (see also \cite[Theorem~1.2]{Ari16} for an earlier iteration). We next give a few examples which show that at least for some graphs our bound \eqref{eq:firstapplication} can be better than \eqref{eq:hkmp20d} and even \eqref{eq:bkkm}; in the next subsection, we will turn to the corresponding lower bounds.

\begin{figure}
\centering
\begin{subfigure}{.33\textwidth}
\centering
\includegraphics{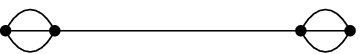}
\end{subfigure}
\begin{subfigure}{.33\textwidth}
\centering
\includegraphics{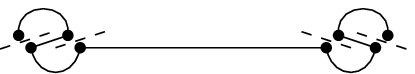}
\end{subfigure}
\caption{A graph given by two $3$-pumpkins connected by an edge. The graph admits an Eulerian path seen on the right side.}
\label{fig:unfolding}
\end{figure}

\begin{example}
We consider the \emph{pumpkin dumbbell} depicted in Figure~\ref{fig:unfolding}, consisting of two $3$-pumpkins connected by an edge (interestingly, the relative edge lengths are irrelevant for these bounds). Then by Corollary~\ref{thm:firstapplication} we have $\doptenergy[k](\Graph)\le \frac{\pi^2}{L^2} (k+4)^2$ for all $k \geq 1$, while since $|\EdgeSet|=7$ and $|\NeuSet|=0$ the upper bound in \eqref{eq:hkmp20d} reads $\doptenergy[k](\Graph)\le \frac{\pi^2}{L^2} (k+7)^2$ (for sufficiently large $k$). Introducing thicker pumpkins would lead to the same conclusion, that \eqref{eq:firstapplication} is better.
\end{example}

\begin{figure}[h!]
\includegraphics{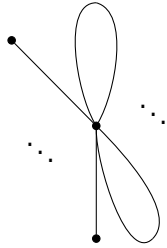}
\caption{Stower graphs as an example of a class of graphs for which \eqref{eq:firstapplication} is better than \eqref{eq:bkkm}}.\label{fig:stower-graphs}
\end{figure}
\begin{example}
The bound on $\mu_k(\Graph)$ in \eqref{eq:firstapplication} is better than \eqref{eq:bkkm} for all flower graphs (where $|\NeuSet|=0$ and $n=1$), and more generally \emph{stower graphs} (flowers with a finite number of pendant edges attached to the central vertex, i.e., a union of a flower and a star; see Figure~\ref{fig:stower-graphs}). These were introduced in \cite{BanLev17}, where they played a major role in the minimization of $\mu_2(\Graph)$ among various classes of graphs. There exist such stower graphs with any $\beta \geq 1$ and $|\NeuSet|\geq 1$ pendant edges, while we certainly have $n \leq \lceil \frac{|\NeuSet|}{2} \rceil$, leading to the assertion that \eqref{eq:firstapplication} is better. Finally, the respective upper bounds coincide for star graphs for which $|\NeuSet|$ is even, since then $\beta = 0$ and $n=\frac{|\NeuSet|}{2}$.
\end{example}

\begin{figure}[h!]
\includegraphics[scale=.5]{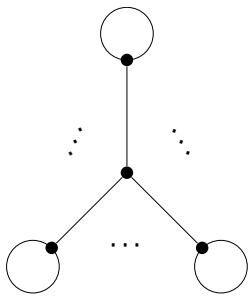}
\caption{Windmill graphs are examples of graphs for which the upper estimate in \eqref{eq:bkkm} is attained. As it turns out this is also the case in \eqref{eq:firstapplication} when the graph has an even number of pendant lassos.}
\label{fig:windmill}
\end{figure}

\begin{example}
Let $\Graph$ be a \emph{windmill graph} $\mathcal W^{2m}$, $m \geq 1$, which consists of $2m$ lassos(\emph{blades}) glued together at a central vertex (see Figure~\ref{fig:windmill}); we assume that all the loops have a common length $\ell > 0$ and the bridges a common length $s>0$. It was shown in \cite{KuSe18} that, if the ratio $\ell/s = 4$, then there is equality in \eqref{eq:bkkm} for a sequence of eigenvalues, for any number of blades. In particular, since $\beta = 2m$,
\begin{equation}
\label{eq:sharpestimatewindmill}
	\mu_k(\mathcal W^{2m})= \frac{\pi^2}{L^2} (k+ 3m-2)^2
\end{equation}
whenever $k=2m(1+5j)+2$ for some $j \geq 1$. Note that $\mathcal W^{2m}$ can be partitioned into $n=m$ clusters, each consisting of exactly two blades glued together (like the dumbbell pumpkin of Figure~\ref{fig:unfolding} but with loops in place of the $3$-pumpkins). This means that the upper bound in \eqref{eq:firstapplication} is also equal to $\frac{\pi^2}{L^2} (k+ 3m-2)^2$; hence we have equality everywhere,
\begin{equation}
	\mu_k(\mathcal W^{2m}) = \doptenergy[k](\mathcal W^{2m}) = \noptenergylax[k-1+\beta](\mathcal W^{2m}) = \frac{\pi^2}{L^2} (k+ 3m-2)^2
\end{equation}
for all $k\geq 1$ of the form $k=2m(1+5j)+2$, $j \geq 1$. In particular, Theorem~\ref{thm:intromain1} is sharp for a subsequence of the eigenvalues of the windmill graph $\mathcal W^{2m}$ when $\ell/s=4$.
\end{example}

Very recently, all graphs which attain the upper estimate in \eqref{eq:bkkm} were classified in \cite{Ser21}. We leave it as an open question whether similar results can be shown for the inequalities from this paper, especially for Theorem~\ref{thm:intromain1}.

\subsection{Lower bounds on $\doptenergy[k](\Graph)$} We start by giving the short proof of Corollary~\ref{cor:secondapplication}; for convenience we recall the estimate to be proved \eqref{eq:secondapplication}
\begin{displaymath}
	\doptenergy[k](\Graph) \geq \frac{\pi^2}{L^2}(k+1 - \beta - |\NeuSet|)^2.
\end{displaymath}

\begin{proof}[Proof of Corollary~\ref{cor:secondapplication}]
By Theorem~\ref{thm:intromain2} and the (sharp) estimate $\noptenergylax[k](\Graph) \geq \frac{\pi^2k^2}{L^2}$ ($k\geq 1$) of \cite[Theorem~3.1]{HKMP20},
\begin{displaymath}
	\doptenergy[k](\Graph) \geq \noptenergylax[k+1-\beta-|\NeuSet|](\Graph) \geq \frac{\pi^2}{L^2}(k+1 - \beta - |\NeuSet|)^2
\end{displaymath}
for all $k\geq \beta + |\NeuSet|$.
\end{proof}

Our estimate \eqref{eq:secondapplication} may compared with the direct lower bound of \cite[Corollary~6.1]{HKMP20},
\begin{equation}
\label{eq:lower-d-theirs}
	\doptenergy[k](\Graph)\geq \frac{\pi^2}{4k L^2} \left ( k^3 + 3  (k - \beta - |\NeuSet|)^3 \right ),
\end{equation}
which holds for sufficiently large $k \geq \beta + |\NeuSet|$. We note that the coefficient of $k$ in the bound in \eqref{eq:secondapplication} is $2(1-\beta-|\NeuSet|)\frac{\pi^2}{L^2}$, to be compared with $-\frac{9}{4}(\beta + |\NeuSet)|\frac{\pi^2}{L^2}$ in \eqref{eq:lower-d-theirs}. Thus \eqref{eq:secondapplication} is \emph{always} better than \eqref{eq:lower-d-theirs}. Other, better, lower bounds in \cite{HKMP20} (e.g. Corollary~6.1, where the factor of $\beta$ is suppressed in \eqref{eq:lower-d-theirs}) are only available for much larger $k$, and are still worse than \eqref{eq:secondapplication} for large classes of graphs, including trees.

\appendix
\section{Rigid partitions}
\label{appendix:classes}

In the current work we always considered general (``connected'') partitions of graphs, where cuts could be made anywhere, including in the middle of the partition elements (the clusters). This was natural for the theory developed: when creating partitions by cutting along eigenfunctions, as was central to the proofs of Theorems~\ref{thm:intromain1} and~\ref{thm:intromain2}, \textit{a priori} one has no control over where one is cutting.

On the other hand, inspired by the situation on domains, a second, more restricted class of partitions was also introduced in \cite{KeKuLeMu20}, where the graph can only be cut at the boundary of the partition clusters and not in their interior. This corresponds to the idea that when partitioning a domain $\Omega$ into pieces $\Omega_i$, there should be no ``extra'' incisions made in the interior of the $\Omega_i$; that is, $\partial\Omega_i$ should correspond exactly to where $\Omega_i$ meets the other $\Omega_j$. This second class may be defined as follows.

\begin{definition}
\label{def:rigid}
We say a (connected) $k$-partition $\Partition=(\Graph_1, \ldots, \Graph_k)$ of $\Graph$ is \emph{rigid} if its boundary set $\partial\Partition$ (see Definition~\ref{def:cutandboundary}) coincides with the cut set $\CutSet(\Graph_\Partition:\Graph)$ (see Definition~\ref{def:cut}). We denote the set of all exhaustive rigid $k$-partitions of $\Graph$ by $\mathfrak{R}_k = \mathfrak{R}_k (\Graph)$.
\end{definition}

\begin{figure}
	\begin{minipage}{.35\textwidth}\includegraphics{examplegraph.eps} \end{minipage}\hspace{3em}
	\begin{minipage}{.35\textwidth}\includegraphics{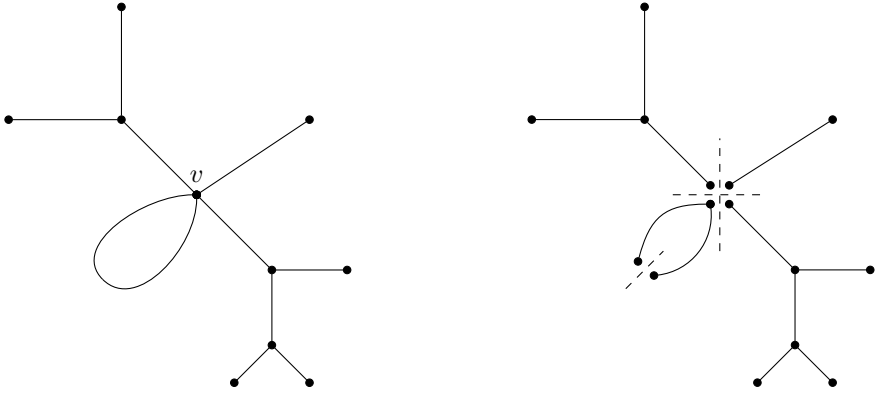} \end{minipage}
	\caption{An example of a non-rigid 4-partition (right) of the graph from Figure~\ref{fig:cutgraphsexamples}(a) (reproduced here, left). The partition is not rigid due to the additional cut in the middle of the loop.}
\end{figure}

The corresponding minimization problems were denoted as follows in \cite[Section~4]{KeKuLeMu20}:
\begin{displaymath}
	\mathcal{L}^{N,r}_{k,p}(\Graph) :=\begin{cases} \inf_{\parti = (\Graph_1,\ldots,\Graph_k) \in \mathfrak{R}_k}\left(\frac{1}{k}\sum\limits_{i=1}^k \mu_2(\Graph_i)^p\right)^{1/p}
	\qquad &\text{if } p \in (0,\infty),\\ \inf_{\parti = (\Graph_1,\ldots,\Graph_k) \in \mathfrak{R}_k}\max\limits_{i=1,\ldots,k} \mu_2(\Graph_i) 
	\qquad &\text{if } p = \infty, \end{cases}
\end{displaymath}
and
\begin{displaymath}
	\mathcal{L}^{N,c}_{k,p}(\Graph) :=\begin{cases} \inf_{\parti = (\Graph_1,\ldots,\Graph_k) \in \mathfrak{C}_k}\left(\frac{1}{k}\sum\limits_{i=1}^k \mu_2(\Graph_i)^p\right)^{1/p}
	\qquad &\text{if } p \in (0,\infty),\\ \inf_{\parti = (\Graph_1,\ldots,\Graph_k) \in \mathfrak{C}_k}\max\limits_{i=1,\ldots,k} \mu_2(\Graph_i) 
	\qquad &\text{if } p = \infty. \end{cases}
\end{displaymath}
Thus our $\noptenergylax[k] (\Graph)$ equals $\mathcal{L}^{N,c}_{k,\infty}(\Graph)$ in \cite{KeKuLeMu20}. Clearly $\mathfrak{R}_k \subset \mathfrak{C}_k$, whence $\mathcal{L}^{N,r}_{k,p}(\Graph) \geq \mathcal{L}^{N,c}_{k,p}(\Graph)$ always (in the Dirichlet case, minimizing over $\mathfrak{R}_k$ or $\mathfrak{C}_k$ makes no difference \cite[Lemma~4.3]{KeKuLeMu20}, justifying the common notation $\mathcal{L}^D_{k,p}(\Graph)$).

We briefly summarize how our Theorems~\ref{thm:intromain1} and~\ref{thm:intromain2} can be adapted to rigid partitions: as noted, since $\mathcal{L}^{N,r}_{k,\infty}(\Graph) \geq \noptenergylax[k] (\Graph)$, Theorem~\ref{thm:intromain1} can be obtained for free:
\begin{equation}
\label{eq:rigid-intromain1}
	\mathcal{L}^{N,r}_{k,\infty}(\Graph) \geq \mathcal{L}^{D}_{k+1-\beta}(\Graph)
\end{equation}
for all $k \geq \max\{\beta,1\}$. It is not at all clear whether Theorem~\ref{thm:intromain2} should hold in the rigid case; here, we merely observe that it will hold for $k$ large enough since, by \cite[Theorem~3.3]{HKMP20}, $\mathcal{L}^{N,r}_{k,\infty}(\Graph) = \noptenergylax[k] (\Graph)$ for sufficiently large $k \geq 1$ (how large depending on the graph): there exists some $k_0 = k_0 (\Graph) \geq \beta + |\NeuSet|$ such that
\begin{equation}
\label{eq:rigid-intromain2}
	\mathcal{L}^{D}_{k}(\Graph) \geq \mathcal{L}^{N,r}_{k+1-\beta-|\NeuSet|,\infty}(\Graph)
\end{equation}
for all $k \geq k_0$. In principle it would be possible to estimate $k_0$ explicitly in terms of explicit features of the graph, but we expect the effort necessary to do so would be disproportionate to its value.

\bibliographystyle{plain}

\end{document}